\documentclass[11pt]{article}
 \usepackage{ulem}

\usepackage[utf8]{inputenc}
\usepackage{stmaryrd}
\usepackage{amsmath,amsthm,verbatim,amssymb,amsfonts,amscd,graphicx}
\usepackage{graphics}
\usepackage{ulem}
\theoremstyle{plain}
\newtheorem{theorem}{Theorem}

\newtheorem{lemma}{Lemma}
\newtheorem{proposition}{Proposition}

\theoremstyle{definition}
\newtheorem{definition}{Definition}
\usepackage{tikz}
\usepackage[fancy]{tikz-inet}
\usetikzlibrary{automata}
\usetikzlibrary{arrows}
\usetikzlibrary{shapes}
\usetikzlibrary{decorations.pathmorphing}
\usetikzlibrary{fit}
\usetikzlibrary{patterns,patterns.meta}
\newcommand{\monoide}[2][E]{
	{\langle\boxempty,\Diamond\mid\{\boxempty^2=\boxempty, \Diamond^2=\varepsilon ,\mathrm{#2}\}\cup \mathrm{#1}\rangle}
	}
\newcommand{\mon}[1]{
	{\langle\boxempty,\Diamond\mid\{\boxempty^2=\boxempty, \Diamond^2=\varepsilon ,\mathrm{#1}\}\rangle}
	}		
\def\M{\mathbb{M}}

\def\Eq{\mathbb{E}\mathrm{q}}
\def\E{\mathbb{E}}
\def\eq{\mathrm{eq}}

\def\Id{\mathrm{Id} }
\title{Presentation of monoids generated by a projection and an involution}
\author{
    Pascal Caron
    \and Jean-Gabriel Luque
    \and Bruno Patrou
}
\begin{document}
\maketitle

\begin{abstract}

Monoids generated by  elements of order two appear in numerous places in the literature. For example, Coxeter reflection groups in geometry, Kuratowski monoids in topology, various monoids
generated by regular operations in language theory and so on. In order to initiate a classification of these monoids, we are interested in the subproblem of monoids, called strict 2-PIMs,
generated by an involution and an idempotent.
In this case we show, when the monoid is finite,  that it is generated by a single equation (in addition to the two defining the involution and the idempotent). We then describe the exact possible forms of this equation and classify them. We recover Kuratowski's theorem as a special case of our study.
\end{abstract}

\section{Introduction}
Kuratowski's theorem \cite{Kur22} is a topology result which states that in a topological space, the monoid generated by a closure operation and the complement has at most 14 elements.
Moreover, there exists a topological space and a set for which this bound of 14 is reached. 
For more on Kuratowski-type theorems for topological spaces, see Gardner and Jackson \cite{GJ08a}.
Hammer \cite{Ham60} noticed that such a statement is valid in a more general setting; it is not necessary to consider topological spaces and topological closure. The theorem is also valid if, instead of topological closure, we use a closure operator on a set X.
Since Kleene's star is a closure operation on formal languages, Peleg \cite{Pel84} was the first to extend Kuratowski's result to operations in this context.

An orbit is a set of values corresponding to the different numbers of sets that can be obtained from operations. Dassow refines Kuratowski's work by taking (formal) languages as a set and involution and closure as operations \cite{Das19,Das21}, then extends his work to Boolean functions \cite{Das22} . In the same framework and also for languages, Charlier \textit{et al.} \cite{CDHS11} show that for a large number of operations, they also obtain finite orbits. 
D. Sherman \cite{She10} generalizes Kuratowski's theorem by considering, in addition to closures and complement, interior union and intersection operations. \\

Idempotency is a mathematical property of a function which states that the successive application of this function to its parameter is equal to the application once. This property arises both in the theory of projectors and in that of closure operations. 

An involution, involutory function, or self-inverse function is a function  that is its own inverse. This notion applies equally to projective geometry, where the square is the identity transformation, and to algebra, where an involution of an algebraic variety is an automorphism of the variety of order two.\\

 An operation that generates a two-element monoid is either an involution or an idempotent operation. While the groups generated by only involutions have been extensively studied (see \textit{e.g.} Coxeter \cite{Cox34,Cox35}), adding idempotent operations in the generator sets gives rise to more complicated structures that have not been classified yet.

In the same vein, we propose to include Kuratowski's theorem in a more general study: monoids generated by exactly one involution, not necessarily the complement, and one idempotent operation, not necessarily a closure one.


\section{Preliminaries}

A \emph{monoid} is a set closed under an associative binary product having a neutral element denoted by $\mathrm{\emph{Id}}$. The juxtaposition of two elements represents their product, the operation symbol being omitted. A presentation of a monoid is a couple $\langle\cal{G}\mid\cal{R}\rangle$ where $\mathcal{G}$ is a set of generators and $\mathcal{R}$ a set of equations inducing an equivalence relation between elements of the monoid. 
The sets $\mathcal{G}$ and $\mathcal{R}$ should be infinite but for this paper we restrict ourselves to finite sets and if $\mathcal{R}=\emptyset$ we recover the free monoid $\mathcal{G}^*$. The presented monoid $\mathrm{M}=\langle \mathcal{G},\mathcal{R}\rangle$ is isomorphic to the quotient of $\mathcal{G}^*$ by the congruence induced by the equations $\mathcal{R}$, i.e. $\mathrm{M}=\mathcal{G}^*/\equiv_{\mathcal{R}}$. For example, the presentation $\langle a, b \mid ab = ba\rangle$ designates the commutative monoid over the alphabet $\{a, b\}$ where each element is equivalent to one of the form $a^pb^q$ ($p, q\in\mathbb{N}$). To lighten the notations, the braces of $\mathcal{G}$ and $\mathcal{R}$ are omitted. A \emph{word} $w$ is an element of $\mathcal{G}^*$. For any $a\in \mathcal{G}$, we denote by $|w|_a$ the number of $a$ in $w$. The \emph{order} of a monoid is its number of elements.

A monoid generated by only one element, $\boxempty$, is said \emph{monogenic}. All infinite monogenic monoids are isomorphic to the semigroup of  natural integers equipped with the addition. The presentation of such a monoid is reduced to  $\langle\boxempty\mid\rangle$ and its elements are $\Id , \boxempty, \boxempty^2, \boxempty^3,\cdots$. For any integer $n\in\mathbb{N}$ there exists, up to an isomorphism, exactly $n$ monogenic monoids of order $n$. Their presentations are of the form $\langle\boxempty\mid\boxempty^n=\boxempty^k\rangle$, with $k\in\{0,\cdots,n-1\}$ and their extensive description is $\{\Id , \boxempty, \boxempty^2,\cdots, \boxempty^{n-1}\}$.  The \emph{order} of an element  $\boxempty$ is the order of the monoid generated by  $\boxempty$.

The structure of monoids generated by two elements (\emph{digenic} monoids) is much more complicated. In this paper, we focus on digenic monoids generated by a set of elements of  order two. Such generators are either \emph{projections} (idempotent operators $\boxempty$ such that $\boxempty^2=\boxempty$) or \emph{involutions} (operators $\Diamond$ such that $\Diamond^2=\Id $). We, then, call \emph{PIM} (for Projection-Involution-Monoid) any monoid generated by only  generators of order $2$ and \emph{2-PIM} any monoid generated by exactly two  generators of order $2$.

If both generators, $\boxempty$ and $\Diamond$, are involutive then their roles are symmetrical and two situations can occur.

\begin{enumerate}
\item There exists a smallest integer $p$ such that $(\boxempty\Diamond)^p=\Id $. In fact, if it is not the case, we may have
 $(\boxempty\Diamond)^p\boxempty= \Id $ then by multiplying both sides by $\Diamond$, we have $\Diamond(\boxempty\Diamond)^p\boxempty\Diamond=\Id $. Then multiplying these two equalities, we have $(\boxempty\Diamond)^{2p+2}=\Id $. We have also $(\boxempty\Diamond)^p\boxempty=\Id $ implies $(\boxempty\Diamond)^p=\boxempty$ and thus $(\boxempty\Diamond)^{2p}=\Id $. Then from the previous equations, we have $(\boxempty\Diamond)^2=\Id $. Thus if $p$ is even then $\boxempty=\Id $ and if $p$ is odd then $\Diamond=\Id $ and so this case is degenerated.

 In the non-degenerated case, the  monoid  
$$\langle\boxempty, \Diamond\mid \boxempty^2=\varepsilon , \Diamond^2=\varepsilon , (\boxempty\Diamond)^p=\varepsilon \rangle,$$
is the \textsc{Coxeter} group $\mathrm{I}_2(p)$ (see \cite{Cox34})
which contains exactly $2p$ elements.
\item For all integer $p$ we have $(\boxempty\Diamond)^p\neq\Id $.\\
Then the generated monoid is infinite and its presentation is:
$$\langle\boxempty, \Diamond\mid \boxempty^2=\varepsilon , \Diamond^2=\varepsilon \rangle.$$
\end{enumerate}

When the generators of a 2-PIM are exactly one projection and one involution, we call the monoid a \emph{strict 2-PIM}. This is the structure we study in this paper.




\section{The presentation of a strict 2-PIM}
\subsection{Generic cases}
If the presentation of the monoid is restricted only by both constraints imposed to a monoid to be a strict 2-PIM, we obtain the following infinite monoid:
\[
\langle\boxempty,\Diamond\mid\boxempty^2=\boxempty, \Diamond^2=\varepsilon \rangle = \{\boxempty^d(\Diamond\boxempty)^k\Diamond^f\mid d,f\in\{0,1\}, k\in\mathbb{N}\}
\]
Otherwise, the various equations that can be added to the presentation are listed in  Table \ref{eqDiogenes}. All are obtained from the generic form
\[
\boxempty^{d_0}(\Diamond\boxempty)^{k_0}\Diamond^{f_0} = \boxempty^{d_1}(\Diamond\boxempty)^{k_1}\Diamond^{f_1}
\]
by varying the values of $d_0, f_0, d_1$ and $f_1$ and eliminating obvious symmetries.

\begin{table}[hbt]
\begin{center}
\begin{tabular}{|cccc|l|lll|c|}
\hline
$d_0$ & $f_0$ & $d_1$ & $f_1$ & form of equation & \multicolumn{3}{c|}{additional details} & set\\
\hline
0 & 0 & 0 & 0 & $(\Diamond\boxempty)^{k_0} = (\Diamond\boxempty)^{k_1}$ & $k_0\neq 0$ & & $k_0<k_1$ & $\Eq_0$\\
\hline
0 & 0 & 0 & 1 & $(\Diamond\boxempty)^{k_0} = (\Diamond\boxempty)^{k_1}\Diamond$ & $k_0\neq 0$ & $k_1\neq 0$ & & $\Eq_1$\\
\hline
0 & 0 & 1 & 0 & $(\Diamond\boxempty)^{k_0} = \boxempty(\Diamond\boxempty)^{k_1}$ & $k_0\neq 0$ & & & $\Eq_2$\\
\hline
0 & 0 & 1 & 1 & $(\Diamond\boxempty)^{k_0} = \boxempty(\Diamond\boxempty)^{k_1}\Diamond$ & $k_0\neq 0$ & & & $\Eq_3$\\
\hline
0 & 1 & 0 & 1 & $(\Diamond\boxempty)^{k_0}\Diamond = (\Diamond\boxempty)^{k_1}\Diamond$ & $k_0\neq 0$ & & $k_0<k_1$ & $\Eq_4$\\
\hline
0 & 1 & 1 & 0 & $(\Diamond\boxempty)^{k_0}\Diamond = \boxempty(\Diamond\boxempty)^{k_1}$ & $k_0\neq 0$ & & & $\Eq_5$\\
\hline
0 & 1 & 1 & 1 & $(\Diamond\boxempty)^{k_0}\Diamond = \boxempty(\Diamond\boxempty)^{k_1}\Diamond$ & $k_0\neq 0$ & & & $\Eq_6$\\
\hline
1 & 0 & 1 & 0 & $\boxempty(\Diamond\boxempty)^{k_0} = \boxempty(\Diamond\boxempty)^{k_1}$ & & & $k_0<k_1$ & $\Eq_7$\\
\hline
1 & 0 & 1 & 1 & $\boxempty(\Diamond\boxempty)^{k_0} = \boxempty(\Diamond\boxempty)^{k_1}\Diamond$ & & & & $\Eq_8$\\
\hline
1 & 1 & 1 & 1 & $\boxempty(\Diamond\boxempty)^{k_0}\Diamond = \boxempty(\Diamond\boxempty)^{k_1}\Diamond$ & & & $k_0<k_1$ & $\Eq_9$\\
\hline
\end{tabular}
\caption{Useful equations in strict 2-PIM presentations}\label{eqDiogenes}
\end{center}
\end{table}

The last column of the Table \ref{eqDiogenes} gives a name to the set of equations described by the row by varying $k_0$ and $k_1$, and the set of all these equations is  denoted by $\Eq=\displaystyle\bigcup_{i=0}^{9}\Eq_i$

\subsection{Degenerated cases}
The details given in the third column avoid considering degenerate cases where monoids are in fact monogenic, as shown by Lemma \ref{LemmaCasDegeneres}.
\begin{lemma}\label{LemmaOutil}
For every $k>0$,
$$\begin{array}{llllll}
\textrm{if} & (\Diamond\boxempty)^k=\Id  & (i)&
\textrm{or} & (\Diamond\boxempty)^k\Diamond=\Id  & (ii)
\end{array}$$
then $\boxempty=\Id $.
\end{lemma}

\begin{proof}
\ \\
\begin{minipage}{0.45\linewidth}
$$\begin{array}{llll}
            & (\Diamond\boxempty)^k=\Id  & (i) &\\
\Rightarrow & (\Diamond\boxempty)^k=\boxempty &\\
\Rightarrow & \boxempty=\Id  & &
\end{array}$$
\end{minipage}
\begin{minipage}{0.45\linewidth}

$$\begin{array}{llll}
            & (\Diamond\boxempty)^k\Diamond=\Id  & (ii) &\\
\Rightarrow & (\Diamond\boxempty)^k=\Diamond &\\
\Rightarrow & (\Diamond\boxempty)^k=\Diamond\boxempty &\\
\Rightarrow & \Diamond\boxempty=\Diamond & &\\
\Rightarrow & \boxempty=\Id  & &
\end{array}$$
\end{minipage}

\end{proof}
\begin{lemma}\label{LemmaCasDegeneres}
A strict 2-PIM  is monogenic if one of the following equations appears in its presentation:
\begin{enumerate}
	\item $\Id= (\Diamond\boxempty)^{k}$ for $k>0$;
	\item $(\Diamond\boxempty)^{k} = \Diamond$ or $(\Diamond\Box)^k\Diamond=\Id$ for $k\geq 0$;
	\item $\Id = \boxempty(\Diamond\boxempty)^{k}$ for $k\geq0$;
	\item $\Id= \boxempty(\Diamond\boxempty)^{k}\Diamond$ with $k\geq0$ ;
	\item $\Diamond = (\Diamond\boxempty)^{k}\Diamond$ with $k>0$ ;
	\item $\Diamond = \boxempty(\Diamond\boxempty)^{k}$ with $k\geq0$ ;
	\item $\Diamond = \boxempty(\Diamond\boxempty)^{k}\Diamond$ with $k\geq0$.
\end{enumerate}
\end{lemma}

\begin{proof}
\noindent
\begin{enumerate}
	\item  Lemma \ref{LemmaOutil} (i) leads to an immediate conclusion.
	\item 
 If $k>0$ then both equalities are equivalent to Lemma \ref{LemmaOutil} (ii). If $k=0$ then the statement reads $\Diamond=\Id$.
	\item 
 We have $\Box=\Box \Id=\Box \Box(\Diamond \Box)^k=\Box(\Diamond \Box)^k=\Id$.
	\item 
 We have $\Id=\Diamond\Diamond=\Diamond\Box(\Diamond \Box)^k\Diamond\Diamond=(\Diamond\Box)^{k+1}$. Hence, the result is recovered from  Item 1.
	\item From $\Diamond=(\Diamond\boxempty)^{k}\Diamond$, we have $\Id =(\Diamond\boxempty)^{k}$ and we conclude using  Item 1.
	\item From $\Diamond = \boxempty(\Diamond\boxempty)^{k}$, we have $\Id = (\Diamond\boxempty)^{k+1}$ and we conclude  using  Item 1.
	\item From $\Diamond = \boxempty(\Diamond\boxempty)^{k}\Diamond$, we have $\Id  = (\Diamond\boxempty)^{k+1}\Diamond$ and we conclude using Item 2.
\end{enumerate}
\end{proof}
The  cases enumerated in Lemma \ref{LemmaCasDegeneres} correspond to degenerated cases of Table \ref{eqDiogenes} for $k_0=0$. We summarize this in Table  \ref{TabDeg} below.
\begin{table}[h]
\[
\begin{array}{|c|c|c|}
\hline \mbox{Point in Lemma \ref{LemmaCasDegeneres}}&
\mbox{Type of equation (Table \ref{eqDiogenes})}&\mbox{Degenerated specialization}\\\hline
1&\Eq_0&k_0=0\\
2&\Eq_1&k_0=0\mbox{ or }k_1=0\\
3&\Eq_2&k_0=0\\
4&\Eq_3&k_0=0\\
5&\Eq_4&k_0=0\mbox{ or }k_1=0\\
6&\Eq_5&k_0=0\\
7&\Eq_6&k_0=0.\\\hline
\end{array}
\]
\caption{Degenerated cases of Table \ref{eqDiogenes}\label{TabDeg}}
\end{table}
We notice that setting $k_0=0$ or $k_1=0$ in equations of $\Eq_7, \Eq_8$ or $\Eq_9$ does not provide any additional degenerated cases. Indeed, in all that cases the number of $\Box$ on either side of the equations is non zero, and so neither side can be reduced to $\Id$ by iterating the rule.
\subsection{Equivalent sets of equations}

We now observe that some sets of equations are redundant, which reduces the number of cases to consider. We say that $\Eq_i$ \emph{induces} $\Eq_{j}$, $i,j\in[0,9]$, if for any $\eq'\in\Eq_j$, there exists $\eq\in\Eq_i$  which is equivalent to $\eq'$. If $\Eq_i$ induces $\Eq_j$ and $\Eq_j$ induces $\Eq_i$, we say that $\Eq_i$ and $\Eq_j$ are  \emph{exchangeable}.

Left or right multiplication  by the idempotent $\Diamond$ induces relations of exchangeability. For instance, left multiplying  by $\Diamond$ any (non degenerated) equation $(\Diamond \Box)^{k_0}=(\Diamond \Box)^{k_1}\Diamond$ in $\Eq_1$ gives 
$\Diamond(\Diamond \Box)^{k_0}=\Diamond(\Diamond \Box)^{k_1}\Diamond$ and equivalently  $\Box(\Diamond \Box)^{k_0-1}=\Box(\Diamond \Box)^{k_1-1}\Diamond$ which is an equation of $\Eq_8$. Conversely, left multiplying by $\Diamond$ any equation of $\Eq_8$ gives an equation of $\Eq_1$. Hence, the set $\Eq_1$ and $\Eq_8$ are exchangeable. Other relations of exchangeability are computed in the same way. These relations are summarized in the following graph, which shows $4$ strongly connected components, illustrating equivalence by exchangeability.

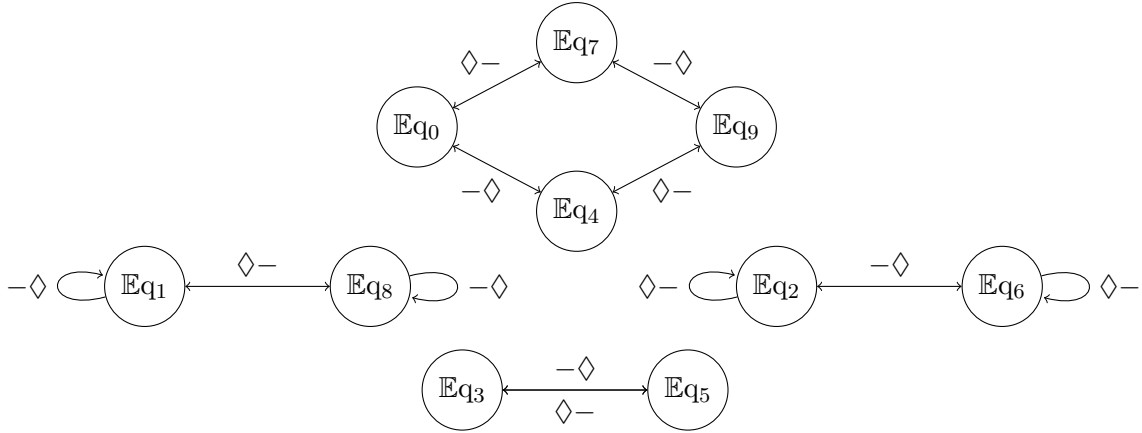
\centerline{
\begin{tikzpicture}[node distance=3cm, bend angle=15]
\node(q0)[state]{$\Eq_0$};
\node(q7)[state,above right of = q0,yshift=-1cm]{$\Eq_7$};
\node(q4)[state,below right of = q0,yshift=1cm]{$\Eq_4$};
\node(q9)[state,below right of = q7,yshift=1cm]{$\Eq_9$};
\path[<->]
(q0)edge[]node[above left,xshift=.2cm]{$\Diamond-$}(q7)
(q0)edge[]node[below left,xshift=.2cm]{$-\Diamond$}(q4)
(q7)edge[]node[above right,xshift=-.2cm]{$-\Diamond$}(q9)
(q4)edge[]node[below right,xshift=-.2cm]{$\Diamond-$}(q9);
\node(q1)[state,below left of = q0,xshift=-1.5cm,yshift=0cm]{$\Eq_1$};
\node(q8)[state,right of = q1]{$\Eq_8$};
\path[<->]
(q1)edge[]node[above]{$\Diamond-$}(q8)
(q1)edge[loop left]node[]{$-\Diamond$}(q1)
(q8)edge[loop right]node[]{$-\Diamond$}(q8);
\node(q2)[state,right of = q8,xshift=2.4cm]{$\Eq_2$};
\node(q6)[state,right of = q2]{$\Eq_6$};
\path[<->]
(q2)edge[]node[above]{$-\Diamond$}(q6)
(q2)edge[loop left]node[]{$\Diamond-$}(q2)
(q6)edge[loop right]node[]{$\Diamond-$}(q6);
\node(q3)[state,below of = q0,xshift=.6cm,yshift=-.5cm]{$\Eq_3$};
\node(q5)[state,right of = q3]{$\Eq_5$};
\path[<->]
(q3)edge[]node[above]{$-\Diamond$}(q5)
(q3)edge[]node[below]{$\Diamond-$}(q5);
\end{tikzpicture}}

An edge of the form $\Eq_i \stackrel{\Diamond-}{\longrightarrow} \Eq_j$ (resp. $\Eq_i \stackrel{-\Diamond}{\longrightarrow} \Eq_j$) means that for any equation $u'=v'$ in $\Eq_j$, the equation $\Diamond u'=\Diamond v'$ (resp. $u'\Diamond = v'\Diamond$) is equivalent to an equation of $\Eq_i$. In other words $\Eq_i$ induces $\Eq_j$.
Every edge being two-way, all the sets of a same component are pairwise exchangeable. As a consequence, we restrict our study to the  $\Eq_0$, $\Eq_1$, $\Eq_2$ and $\Eq_3$ sets only.

\subsection{A simple parameterization}

The sets $\Eq_0$, $\Eq_1$, $\Eq_2$, $\Eq_3$ are described by equations with two integer parameters $k_0$ and $k_1$. For $\Eq_1$, $\Eq_2$ and $\Eq_3$, we show that only one integer parameter and one boolean parameter are needed to describe the equations. To do this, we define four subsets ; each is partitionned into two parts corresponding to the boolean parameter. We take this opportunity to present most of the equations in a slightly different form, and to index our sets using a binary notation that will be useful later on. Note that the sets $\E_{00}$ and $\Eq_0$ are identical, the notation $\E_{00}$ being introduced for the sake of homogeneity with $\E_i$, $i\in\{01,10,11\}$. In the following, we   denote by $\underline{i}=2i_1+i_0$ the decimal notation for $i=i_1i_0$. Let us also define  $\E=\displaystyle\bigcup_{i\in\{00,01,10,11\}}\E_i$.


\begin{definition}
Let $\E_{00}=\E_{00}^\circ\cup \E_{00}^\bullet$ with
$$\begin{array}{lcl}
\E_{00}^{\circ} &=& \{(\Diamond\boxempty)^{k}=(\Diamond\boxempty)^{k+\ell}|k,\ell>0,\ell\text{ even}\}\\
\E_{00}^{\bullet} &=& \{(\Diamond\boxempty)^{k}=(\Diamond\boxempty)^{k+\ell}|k,\ell>0, \ell\text{ odd}\}
\end{array}$$
\end{definition}

\begin{definition}
Let $\E_i=\E_i^\circ \cup \E_i^\bullet$, $i\in\{01,10,11\}$ with 
$$\begin{array}{lcl}
\E_{01}^\circ &=& \{(\Diamond\boxempty)^{k}=(\Diamond\boxempty)^{k+1}\Diamond|k>0\}\\
\E_{01}^\bullet &=& \{(\Diamond\boxempty)^{k}=(\Diamond\boxempty)^{k}\Diamond|k>0\}\\
\E_{10}^\circ &=& \{(\Diamond\boxempty)^{k}=(\boxempty\Diamond)^{k}\boxempty|k>0\}\\
\E_{10}^{\bullet} &=& \{(\Diamond\boxempty)^{k}=(\boxempty\Diamond)^{k-1}\boxempty|k>0\}\\
\E_{11}^\circ &=& \{(\Diamond\boxempty)^{k}=(\boxempty\Diamond)^{k}\mid k>0\}\\
\E_{11}^{\bullet} &=& \{(\Diamond\boxempty)^{k}=(\boxempty\Diamond)^{k+1}\mid k>0\}
\end{array}$$
\end{definition}

Hereafter, we denote $\eq_i(k_0,k_1)$ the equation of $\Eq_{\underline{i}}$ parameterized by $k_0$ and $k_1$, with $i\in \{01,10,11\}$, and $\eq_{00}(k,\ell)$ the equation $(\Diamond\boxempty)^k=(\Diamond\boxempty)^{k+\ell}$ of $\E_{00}$. We also denote by $\eq_i^s(k)$ the equation of $\E_i^s$ (resp. $\E_i$, if $s=\varepsilon$), $i\in\{01,10,11\}$, $s\in \{\circ,\bullet,\varepsilon\}$, of parameter $k$.

\begin{lemma}\label{lm-impair}

For each $i\in\{01,10,11\}$, we have $\eq_i^\bullet(k) \Rightarrow \eq_{00}(k,1)$. 
\end{lemma}
\begin{proof}
For $i=01$, we obtain the result, 
by right multiplying $\eq_{01}^\bullet(k)$ by $\boxempty$.

For $i=10$, 
, we obtain the result by left multiplying $\eq_{10}^\bullet(k)$ by  $\Diamond\boxempty\Diamond$.

By left and right multiplying by $\Diamond$, the equality $\eq_{11}^\bullet(k):(\Diamond\boxempty)^{k}=(\boxempty\Diamond)^{k+1}$ is equivalent to  
\begin{equation}\label{eq4-2}
(\boxempty\Diamond)^{k}=(\Diamond\boxempty)^{k+1}.
\end{equation}
We then have
$$\begin{array}{rcll}
(\boxempty\Diamond)^{k}&=&(\Diamond\boxempty)^{k}\Diamond\boxempty&\text{from }(\ref{eq4-2})\\
&=&(\boxempty\Diamond)^{k+1}\Diamond\boxempty=(\boxempty\Diamond)^{k}\boxempty=\boxempty(\Diamond\boxempty)^k&\text{from }\eq_{11}^\bullet(k)\\
&=&\boxempty(\boxempty\Diamond)^{k+1}=(\boxempty\Diamond)^{k+1}&\text{from }\eq_{11}^\bullet(k).\end{array}$$
Left and right multiplying by $\Diamond$ the equation $(\boxempty\Diamond)^{k}=(\boxempty\Diamond)^{k+1}$
we obtain $(\Diamond\boxempty)^{k}=(\Diamond\boxempty)^{k+1}$.
\end{proof}
\begin{lemma}\label{lm-pair}
For each $i\in\{01,10,11\}$, we have $\eq_i^\circ(k) \Rightarrow \eq_{00}(k,2)$.
\end{lemma}
\begin{proof}
From $\eq_{01}^\circ(k)
$, we obtain the result by right multiplying by $\boxempty$.

From $\eq_{10}^\circ(k)
$, we obtain the result by left multiplying by $\Diamond\boxempty\Diamond$.

From $\eq_{11}^\circ(k)
$, we obtain the result by multiplying on the left by $\Diamond\boxempty\Diamond$ and on the right by $\boxempty$.
\end{proof}

\begin{lemma}\label{lmEqi->Ei}
For each  $i\in\{01,10,11\}$, we have $\eq_i(k_0,k_1)\Rightarrow \eq_i^\circ(k_0)$. 
\end{lemma}
\begin{proof}
Case $i=01$ :
$$\begin{array}{llllll}
            & (\Diamond\boxempty)^{k_0}                          & = & (\Diamond\boxempty)^{k_1}\Diamond                          & (1) & \\
\Rightarrow & (\Diamond\boxempty)^{k_0}\Diamond\boxempty\Diamond & = & (\Diamond\boxempty)^{k_1}\Diamond\Diamond\boxempty\Diamond &     & \\
\Rightarrow & (\Diamond\boxempty)^{k_0+1}\Diamond                & = & (\Diamond\boxempty)^{k_1}\Diamond                          & (2) & \\
\Rightarrow & (\Diamond\boxempty)^{k_0}                          & = & (\Diamond\boxempty)^{k_0+1}\Diamond                        &     & \textrm{from (1) and (2)}
\end{array}$$

Case $i=10$ :
$$\begin{array}{llllll}
            & (\Diamond\boxempty)^{k_0}                           & = & \boxempty(\Diamond\boxempty)^{k_1}                           & (1) & \\
\Rightarrow & (\Diamond\boxempty)^{k_0}\boxempty\Diamond\boxempty & = & \boxempty(\Diamond\boxempty)^{k_1}\boxempty\Diamond\boxempty &     & \\
\Rightarrow & (\Diamond\boxempty)^{k_0+1}                         & = & \boxempty\Diamond\boxempty(\Diamond\boxempty)^{k_1}          &     & \\
\Rightarrow & (\Diamond\boxempty)^{k_0+1}                         & = & \boxempty\Diamond(\Diamond\boxempty)^{k_0}                   &     & \textrm{from (1)}\\
\Rightarrow & \Diamond(\Diamond\boxempty)^{k_0+1}                 & = & \Diamond\boxempty\Diamond(\Diamond\boxempty)^{k_0}           &     & \\
\Rightarrow & (\boxempty\Diamond)^{k_0}\boxempty                  & = & (\Diamond\boxempty)^{k_0}                                    &     &
\end{array}$$

Case $i=11$ :
$$\begin{array}{llllll}
            & (\Diamond\boxempty)^{k_0}                          & = & \boxempty(\Diamond\boxempty)^{k_1}\Diamond                          & (1) & \\
\Rightarrow & \Diamond(\Diamond\boxempty)^{k_0}                  & = & \Diamond\boxempty(\Diamond\boxempty)^{k_1}\Diamond                  &     & \\
\Rightarrow & \Diamond(\Diamond\boxempty)^{k_0}\boxempty\Diamond & = & \Diamond\boxempty(\Diamond\boxempty)^{k_1}\Diamond\boxempty\Diamond &     & \\
\Rightarrow & (\boxempty\Diamond)^{k_0}                          & = & \Diamond\boxempty(\Diamond\boxempty)^{k_1}\Diamond\boxempty\Diamond &     & \\
\Rightarrow & (\boxempty\Diamond)^{k_0}                          & = & \Diamond\boxempty\Diamond\boxempty(\Diamond\boxempty)^{k_1}\Diamond &     & \\
\Rightarrow & (\boxempty\Diamond)^{k_0}                          & = & \Diamond\boxempty\Diamond(\Diamond\boxempty)^{k_0}                  &     & \textrm{from (1)}\\
\Rightarrow & (\boxempty\Diamond)^{k_0}                          & = & (\Diamond\boxempty)^{k_0}                                           &     &
\end{array}$$
\end{proof}

\begin{proposition}\label{lm-Eqi<->Ei}
For each $i\in\{01,10,11\}$, the sets $\Eq_{\underline{i}}$ and $\E_i$ are exchangeable.
\end{proposition}
\begin{proof}
We have $\E_i\subset\Eq_{\underline{i}}$. So $\Eq_{\underline{i}}$ trivially induces $\E_i$. Let us show the counterpart :
\begin{itemize}
\item If $i=01$ then we consider equation $\eq_{01}(k_0,k_1)$: $(\Diamond\boxempty)^{k_0}=(\Diamond\boxempty)^{k_1}\Diamond$. We can assume $k_0\leq k_1$ since, by right multiplying  by $\Diamond$, we obtain the equivalent equation $\eq_{01}(k_1,k_0)$.

There are two cases, depending on the parity of $k_1-k_0$.
\begin{enumerate}
\item If $k_1=k_0+2r+1$ ($r\in\mathbb{N}$): From Lemma \ref{lmEqi->Ei}, we immediately deduce equation $\eq_{01}^\circ(k_0)$.

Conversely, from $\eq_{01}^\circ(k_0)$: $(\Diamond\boxempty)^{k_0}=(\Diamond\boxempty)^{k_0+1}\Diamond$, we deduce $\eq_{00}(k_0,2)$: $(\Diamond\boxempty)^{k_0}=(\Diamond\boxempty)^{k_0+2}$ From Lemma \ref{lm-pair}. We obtain $\eq_{00}(k_0,2r)$: $(\Diamond\boxempty)^{k_0}=(\Diamond\boxempty)^{k_0+2r}$ iterating $r$ times $\eq_{00}(k_0,2)$. Last, combining $\eq_{01}^\circ(k_0)$ and $\eq_{00}(k_0,2r)$, we obtain $\eq_1(k_0,k_1)$: $(\Diamond\boxempty)^{k_0}=(\Diamond\boxempty)^{k_1}\Diamond$.
\item If $k_1=k_0+2r$ ($r\in\mathbb{N}$): once again, by Lemma \ref{lmEqi->Ei}, we obtain equation $\eq_{01}^\circ(k_0)$, then we deduce $\eq_{00}(k_0,2)$ from Lemma \ref{lm-pair}. We still obtain $\eq_{00}(k_0,2r)$ iterating $r$ times $\eq_{00}(k_0,2)$. By combining $\eq_{01}(k_0,k_1)$ and $\eq_{00}(k_0,2r)$, we obtain $\eq_{01}^\bullet(k_0)$: $(\Diamond\boxempty)^{k_0}=(\Diamond\boxempty)^{k_0}\Diamond$.

Conversely, from $\eq_{01}^\bullet(k_0)$, we deduce $\eq_{00}(k_0,1)$: $(\Diamond\boxempty)^{k_0}=(\Diamond\boxempty)^{k_0+1}$ by Lemma \ref{lm-impair}. Iterating $2r$ times $\eq_{00}(k_0,1)$, we obtain $\eq_{00}(k_0,2r)$, which we combine with $\eq_{01}^\bullet(k_0)$ to find again $\eq_{01}(k_0,k_1)$.
\end{enumerate}
\item If $i=10$ then we consider equation $\eq_{10}(k_0,k_1)$: $(\Diamond\boxempty)^{k_0}=(\boxempty\Diamond)^{k_1}\boxempty$. We can assume  $k_0\leq k_1 + 1$ since, by left multiplying by $\Diamond$, we obtain the equivalent equation  $\eq_{10}(k_1+1,k_0-1)$ where $k_0>k_1+1\Rightarrow k_1+1\leq(k_0-1)+1$.

There are two cases depending on the parity of $k_1-k_0$.
\begin{enumerate}
\item If $k_1=k_0+2r+1$ ($r\geq-1$): from Lemma \ref{lmEqi->Ei}, we deduce equation $\eq_{10}^\circ(k_0)$, then we obtain $\eq_{00}(k_0,2)$ by Lemma \ref{lm-pair}. We deduce $\eq_{00}(k_0,2r+2)$ by  iterating $r+1$ times $\eq_{00}(k_0,2)$. By combining $\eq_{10}(k_0,k_1)$ and $\eq_{00}(k_0,2r+2)$, we obtain $\eq_{10}^\bullet(k_0)$: $(\Diamond\boxempty)^{k_0}=(\boxempty\Diamond)^{k_0-1}\boxempty$. Indeed, $(\Diamond\boxempty)^{k_0}=(\boxempty\Diamond)^{k_0+2r+1}\boxempty=\boxempty(\Diamond\boxempty)^{k_0+2r+1}=\boxempty(\Diamond\boxempty)^{(k_0+2r+2)-1}=\boxempty(\Diamond\boxempty)^{k_0-1}=(\boxempty\Diamond)^{k_0-1}\boxempty$.

Conversely, from $\eq_{10}^\bullet(k_0)$, we deduce $\eq_{00}(k_0,1)$: $(\Diamond\boxempty)^{k_0}=(\Diamond\boxempty)^{k_0+1}$ by Lemma \ref{lm-impair}. Iterating $2(r+1)$ times $\eq_{00}(k_0,1)$, we obtain $\eq_{00}(k_0,2(r+1))$, which we combine with $\eq_{10}^\bullet(k_0)$ to find again $\eq_{10}(k_0,k_1)$.
\item If $k_1=k_0+2r$ ($r\in\mathbb{N}$): by Lemma \ref{lmEqi->Ei}, we immediately deduce equation $\eq_{10}^\circ(k_0)$.

Conversely, from $\eq_{10}^\circ(k_0)$: $(\Diamond\boxempty)^{k_0}=(\boxempty\Diamond)^{k_0}\boxempty$, we deduce $\eq_{00}(k_0,2)$ from Lemma \ref{lm-pair}. We obtain $\eq_{00}(k_0,2r)$: $(\Diamond\boxempty)^{k_0}=(\Diamond\boxempty)^{k_0+2r}$ iterating $r$ times $\eq_{00}(k_0,2)$. Last,  combining $\eq_{10}^\circ(k_0)$ and $\eq_{00}(k_0,2r)$, we obtain $\eq_{10}(k_0,k_1)$: $(\Diamond\boxempty)^{k_0}=(\boxempty\Diamond)^{k_1}\boxempty$.
\end{enumerate}

\item If $i=11$ then  we consider equation $\eq_{11}(k_0,k_1)$: $(\Diamond\boxempty)^{k_0}=(\boxempty\Diamond)^{k_1}$. We can assume $k_0\leq k_1$ since, by left and right multiplying   by $\Diamond$, we obtain the equivalent equation $\eq_{11}(k_1,k_0)$.

There are two cases, depending on the parity of $k_1-k_0$.
\begin{enumerate}
\item If $k_1=k_0+2r+1$ ($r\in\mathbb{N}$): from Lemma \ref{lmEqi->Ei}, we obtain equation $\eq_{11}^\circ(k_0)$, then we deduce $\eq_{00}(k_0,2)$ from Lemma \ref{lm-pair}. We obtain $\eq_{00}(k_0,2r)$ iterating $r$ times $\eq_{00}(k_0,2)$. Combining $\eq_{11}(k_0,k_1)$ and $\eq_{00}(k_0,2r)$, we obtain $\eq_{11}^\bullet(k_0)$: $(\Diamond\boxempty)^{k_0}=(\boxempty\Diamond)^{k_0+1}$. Indeed, $(\Diamond\boxempty)^{k_0}=(\boxempty\Diamond)^{k_0+2r+1}=\boxempty(\Diamond\boxempty)^{k_0+2r}\Diamond=\boxempty(\Diamond\boxempty)^{k_0}\Diamond=(\boxempty\Diamond)^{k_0+1}$.

Conversely, from $\eq_{11}^\bullet(k_0)$, we deduce $\eq_{00}(k_0,1)$: $(\Diamond\boxempty)^{k_0}=(\Diamond\boxempty)^{k_0+1}$ by Lemma \ref{lm-impair}. Iterating $2r$ times $\eq_{00}(k_0,1)$, we obtain $\eq_{00}(k_0,2r)$, which we combine with $\eq_{11}^\bullet(k_0)$ to find again $\eq_{11}(k_0,k_1)$. Indeed, $(\Diamond\boxempty)^{k_0}=(\boxempty\Diamond)^{k_0+1}=\boxempty(\Diamond\boxempty)^{k_0}\Diamond=\boxempty(\Diamond\boxempty)^{k_0+2r}\Diamond=(\boxempty\Diamond)^{k_1}$
\item If $k_1=k_0+2r$ ($r\in\mathbb{N}$): from Lemma \ref{lmEqi->Ei}, we immediately obtain equation $\eq_{11}^\circ(k_0)$.

Conversely, from $\eq_{11}^\circ(k_0)$: $(\Diamond\boxempty)^{k_0}=(\boxempty\Diamond)^{k_0}$, we deduce $\eq_{00}(k_0,2)$ by Lemma \ref{lm-pair}. We obtain $\eq_{00}(k_0,2r)$: $(\Diamond\boxempty)^{k_0}=(\Diamond\boxempty)^{k_0+2r}$  iterating $r$ times $\eq_{00}(k_0,2)$, then $\Diamond(\Diamond\boxempty)^{k_0}\Diamond=\Diamond(\Diamond\boxempty)^{k_0+2r}\Diamond$ and $(\boxempty\Diamond)^{k_0}=(\boxempty\Diamond)^{k_0+2r}$. Combining with  $\eq_{11}^\circ(k_0)$, we obtain $\eq_{11}(k_0,k_1)$.
\end{enumerate}
\end{itemize}
\end{proof}

Proposition \ref{lm-Eqi<->Ei} allows us to restrict our study to the sets $\E_i$, $i\in\{00,01,10,11\}$.

\section{Reducing equations}
The aim of this section is to obtain the lattice shown in Figure \ref{latticeEqu}, which highlights the induction relations between the sets of equations.
First, we notice that any conjunction of equations in a set $\E_i^s$ ($i\in\{00,01,10,11\}$, $s\in\{\circ,\bullet,\varepsilon\}$) is equivalent to a single equation in the same set.  Hence we  investigate  the conjunctions of equations belonging to distinct sets.
\subsection{Homogeneous systems}

\begin{lemma}\label{lem:eq1+eq1=eq1}
For every $k_0,k_1,\ell_0,\ell_1>0$, we have 
\[
\eq_{00}(k_0,\ell_0)\wedge \eq_{00}(k_1,\ell_1)\Leftrightarrow \eq_{00}(m,d)
\]
with $m=\min(k_0,k_1)$ and $d=\gcd(\ell_0,\ell_1)$.
\end{lemma}
\begin{proof}
Without restriction, we assume that $k_0\leq k_1$.
Let $0\leq r< \ell_0$ be the only integer such that $k_1-k_0+r\equiv 0\mod \ell_0$.
Then, there exists $k\in\mathbb{N}$ such that $k_1-k_0+r=k\ell_0$.
 Applying  $k$ times $\eq_{00}(k_0,\ell_0)$ , we obtain
\[
(\Diamond\boxempty)^{k_0+\ell_1}=(\Diamond\boxempty)^{k_0+k\ell_0+\ell_1}=(\Diamond\boxempty)^{k_0+k_1-k_0+r+\ell_1}=(\Diamond\boxempty)^{k_1+r+\ell_1}.
\]
Then, considering $\eq_{00}(k_1,\ell_1)$, we obtain
\[
(\Diamond\boxempty)^{k_0+\ell_1}=(\Diamond\boxempty)^{k_1+r+\ell_1}=(\Diamond\boxempty)^{k_1+r}=(\Diamond\boxempty)^{k_0+k_1-k_0+r}=(\Diamond\boxempty)^{k_0+k\ell_0}.
\]
Once again, applying $k$ times $\eq_{00}(k_0,\ell_0)$, we obtain.
\begin{equation}\label{eqk0+l1=k0}
(\Diamond\boxempty)^{k_0+\ell_1}=(\Diamond\boxempty)^{k_0+k\ell_0}=(\Diamond\boxempty)^{k_0}.
\end{equation}
Let us denote by $\alpha_0$ and $\beta_0$ the coefficients of B\'ezout:
\[
\alpha_0 r_0+\alpha_1 r_1=d.
\] 
We assume that $\alpha_0>0$ (resp. $\alpha_1>0$). Applying $\alpha_0$ times $\eq_{00}(k_0,\ell_0)$ (resp. $\alpha_1$ times equation (\ref{eqk0+l1=k0})), we have
\[
(\Diamond\boxempty)^{k_0}=(\Diamond\boxempty)^{k_0+\alpha_0\ell_0} \left(\mbox{resp.}  (\Diamond\boxempty)^{k_0}=(\Diamond\boxempty)^{k_0+\alpha_1\ell_1} \right).
\]
Then,
\[(\Diamond\boxempty)^{k_0+\alpha_0\ell_0}=(\Diamond\boxempty)^{k_0+\alpha_0\ell_0+\alpha_1\ell_1-\alpha_1\ell_1}=(\Diamond\boxempty)^{k_0+d-\alpha_1\ell_1}=(\Diamond\boxempty)^{d+k_0-\alpha_1\ell_1}\stackrel{\mbox{\tiny{by }\ref{eqk0+l1=k0}}}{=}(\Diamond\boxempty)^{d+k_0}\]
\[(\mbox{resp.} (\Diamond\boxempty)^{k_0+\alpha_1\ell_1}=(\Diamond\boxempty)^{k_0+\alpha_0\ell_0+\alpha_1\ell_1-\alpha_0\ell_0}=(\Diamond\boxempty)^{k_0+d-\alpha_0\ell_0}=(\Diamond\boxempty)^{d+k_0-\alpha_0\ell_0}\stackrel{\mbox{\tiny{by $\eq_1(k_0,\ell_0)$}}}{=}(\Diamond\boxempty)^{d+k_0}).\]
In other words,
$
(\Diamond\boxempty)^{k_0}=(\Diamond\boxempty)^{k_0+d}.
$
And this proves the result.\\
Conversely, if we assume that $k_0\leq k_1$ and if we set $q_0=\frac{\ell_0}{d}$ and $q_1=\frac{\ell_1}{d}$, we obtain
\[
(\Diamond\boxempty)^{k_0}=(\Diamond\boxempty)^{k_0+q_0d}=(\Diamond\boxempty)^{k_0+\ell_0}
\]
and
\[
(\Diamond\boxempty)^{k_1}=
(\Diamond\boxempty)^{k_1-k_0}(\Diamond\boxempty)^{k_0}=(\Diamond\boxempty)^{k_1-k_0}(\Diamond\boxempty)^{k_0+q_1d}=(\Diamond\boxempty)^{k_1+\ell_1}.
\]

\end{proof}

\begin{proposition}\label{E00XE00<=>E00}
For any pair of equations $\eq,\eq'$ of $\E_{00}$ there exists an equation $\eq''$ of $\E_{00}$ which is equivalent to $\eq \wedge \eq'$. Moreover $\eq''\in \E_{00}^\circ$ if and only if $\eq\in \E_{00}^\circ$ and $\eq'\in \E_{00}^\circ$.
\end{proposition}
\begin{proof}
It is easy to see that for the $\gcd$ of two numbers to be even, both numbers must be even. Thanks to Lemma \ref{lem:eq1+eq1=eq1}, we  conclude the proof.
\end{proof}

\begin{proposition}\label{k0k1<=>min}
For each $i\in\{01,10,11\}$, $s\in\{\circ,\bullet\}$ and $k_0,k_1>0$, we have 
$$\eq_i^s(k_0)\wedge \eq_i^s(k_1)\Leftrightarrow \eq_i^s(\min(k_0,k_1)).$$
\end{proposition}
\begin{proof}
The implication $$\eq_i^s(k_0)\wedge \eq_i^s(k_1)\Rightarrow \eq_i^s(\min(k_0,k_1))$$ is obvious.\\
Conversely, for $i=01$ (resp. $i=10$), if we set $m=\min(k_0,k_1)$, left (resp. right) multiplying  by $\Diamond\boxempty$ gives $$\eq_i^s(m)\Rightarrow\eq_i^s(m+1)$$ and, repeating the process, for any $k\geq m$, $$\eq_i^s(m)\Rightarrow\eq_i^s(k).$$
When $i=11$, the converse part is similarly obtained by observing :
$$
\begin{array}{rcll}
                 (\Diamond\boxempty)^{m+1} &=& \Diamond(\boxempty\Diamond)^m\boxempty                      & \\
&  = &\Diamond(\Diamond\boxempty)^m\boxempty                      & \mbox{from } \eq_{11}^\circ(m)\\
 &=& (\boxempty\Diamond)^{m-1}\boxempty                          & \\
&=& \boxempty(\boxempty\Diamond)^{m-1}\boxempty\Diamond\Diamond & \\
&=& \boxempty(\boxempty\Diamond)^m\Diamond                      & \\
&=& \boxempty(\Diamond\boxempty)^m\Diamond                      & \mbox{from } \eq_{11}^\circ(m)\\
&=& (\boxempty\Diamond)^{m+1}                                   & 
\end{array}
$$
In other words $\eq_{11}^\circ(m)$ implies $\eq_{11}^\circ(m+1)$ and an almost identical sequence of equivalences shows that $\eq_{11}^\bullet(m)$ implies $\eq_{11}^\bullet(m+1)$.
%
%
%
\end{proof}

Therefore, we summarize the result of the section as follows: A conjunction of several equations belonging to the same set $\E^s_i$   reduces to a single equation. 
\subsection{Heterogeneous systems}
To deal with the case of equations belonging to distinct sets, we begin by observing the graph of implications obtained by direct application of the $\boxempty$ operation.

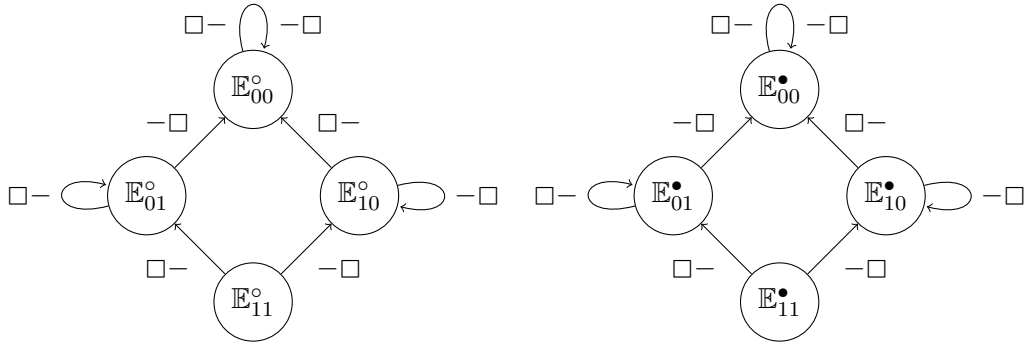
\centerline{
\begin{tikzpicture}[node distance=2cm, bend angle=15]
\node(q00c)[state]{$\E_{00}^{\circ}$};
\node(q01c)[state,below left of = q00c]{$\E_{01}^{\circ}$};
\node(q10c)[state,below right of = q00c]{$\E_{10}^{\circ}$};
\node(q11c)[state,below right of = q01c]{$\E_{11}^{\circ}$};
\path[->]
(q11c)edge[]node[below left]{$\boxempty-$}(q01c)
(q11c)edge[]node[below right]{$-\boxempty$}(q10c)
(q01c)edge[]node[above left]{$-\boxempty$}(q00c)
(q10c)edge[]node[above right]{$\boxempty-$}(q00c)
(q00c)edge[loop above]node[below left,xshift=-.2cm]{$\boxempty-$}node[below right,xshift=.2cm]{$-\boxempty$}(q00c)
(q01c)edge[loop left]node[]{$\boxempty-$}(q01c)
(q10c)edge[loop right]node[]{$-\boxempty$}(q10c);
\node(q00b)[state,right of = q00c,xshift=5cm]{$\E_{00}^{\bullet}$};
\node(q01b)[state,below left of = q00b]{$\E_{01}^{\bullet}$};
\node(q10b)[state,below right of = q00b]{$\E_{10}^{\bullet}$};
\node(q11b)[state,below right of = q01b]{$\E_{11}^{\bullet}$};
\path[->]
(q11b)edge[]node[below left]{$\boxempty-$}(q01b)
(q11b)edge[]node[below right]{$-\boxempty$}(q10b)
(q01b)edge[]node[above left]{$-\boxempty$}(q00b)
(q10b)edge[]node[above right]{$\boxempty-$}(q00b)
(q00b)edge[loop above]node[below left,xshift=-.2cm]{$\boxempty-$}node[below right,xshift=.2cm]{$-\boxempty$}(q00b)
(q01b)edge[loop left]node[]{$\boxempty-$}(q01b)
(q10b)edge[loop right]node[]{$-\boxempty$}(q10b);
\end{tikzpicture}}

Unfortunately, the edges only represent inductions but not exchanges such as we observed for the $\Diamond$ operation.
We therefore need to consider several cases through the lemmas that follow.

\begin{lemma}\label{2et3<=>4}
For any $k_0, k_1\geq 0$, we have
$$\eq_{01}^\circ(k_0)\wedge\eq_{10}^\circ(k_1)\Leftrightarrow\eq_{11}^\circ(\min(k_0,k_1)).$$
\end{lemma}
\begin{proof}
We first assume $k_0\leq k_1$. From $\eq_{01}^\circ(k_0)$, by Lemma \ref{lm-pair}, we deduce $\eq_{00}(k_0,2)$ which is equivalent to $(\boxempty\Diamond)^{k_0}=(\boxempty\Diamond)^{k_0+2}$ ($\dag$) by left and right multiplying by $\Diamond$ both sides of this last equation. As $k_1\geq k_0$, iterating $\eq_{00}(k_0,2)$, we also have $(\Diamond\boxempty)^{k_0}=(\Diamond\boxempty)^{k_1+\varepsilon}$ with $\varepsilon\in \{0,1\}$ ($\ddag$).
Then $\eq_{01}^\circ(k_0)\wedge\eq_{10}^\circ(k_1)$ implies
$$\begin{array}{rcll}
 (\Diamond\boxempty)^{k_0} & = & (\Diamond\boxempty)^{k_1+1+\varepsilon}                                        & \mbox{by }(\ddag)\\
 & = & (\boxempty\Diamond)^{k_1} \boxempty(\Diamond\boxempty)^{1+\varepsilon}\Diamond & \mbox{by } \eq_{10}^\circ(k_1)\\    & = & \boxempty(\Diamond\boxempty)^{k_1} (\Diamond\boxempty)^{1+\varepsilon}\Diamond & \\
					                         & = & \boxempty(\Diamond\boxempty)^{k_0+1}\Diamond                                   & \mbox{by } (\ddag) \\
					                               & = & 
                                        (\boxempty\Diamond)^{k_0+2}                                                    & \\
					                                    & = & (\boxempty\Diamond)^{k_0}                         & \mbox{by } (\dag)
\end{array}$$
In other words, $$\eq_{01}^\circ(k_0)\wedge\eq_{10}^\circ(k_1) \Rightarrow \eq_{11}^\circ(\min(k_0,k_1)).$$

Conversely, we deduce $\eq_{01}^\circ(k_0)$ by left multiplying by $\Diamond\boxempty\Diamond$ both sides of $\eq_{11}^\circ(k_0)$. Indeed
$$\begin{array}{rll}
\Diamond\boxempty\Diamond(\Diamond\boxempty)^{k_0}&=&\Diamond\boxempty\Diamond(\boxempty\Diamond)^{k_0}\\
(\Diamond\boxempty)^{k_0}                					&=&(\Diamond\boxempty)^{k_0+1}\Diamond               
\end{array}$$
We recall that $\eq_{01}^\circ(k_0)$ involves ($\dag$) and $\eq_{00}(k_0,2)$ and then $\eq_{00}(k_0+1,2)$ by right multipliying by $\Diamond\boxempty$.
Iterating $\eq_{00}(k_0+1,2)$, we obtain $(\Diamond\boxempty)^{k_0+\varepsilon}=(\Diamond\boxempty)^{k_1}$ with $\varepsilon\in \{0,1\}$ ($\S$). Then $\eq_{11}^\circ(k_0)$ implies

$$
\begin{array}{rcll}
(\Diamond\boxempty)^{k_1} & = & (\Diamond\boxempty)^{k_0+\varepsilon}                                              &                               \\
& = & (\Diamond\boxempty)^{k_0+1+\varepsilon}\Diamond                                    & \mbox{by } \eq_{01}^\circ(k_0)\\
	& = & (\boxempty\Diamond)^{k_0}(\Diamond\boxempty)^{1+\varepsilon}\Diamond               & \mbox{by } \eq_{11}^\circ(k_0)\\
		 & = & (\boxempty\Diamond)^{k_0}\Diamond\boxempty\Diamond(\boxempty\Diamond)^{\varepsilon}&                               \\
		   & = & (\boxempty\Diamond)^{k_0}(\boxempty\Diamond)^{\varepsilon}                         &                               \\
	 & = & (\boxempty\Diamond)^{k_0+2}(\boxempty\Diamond)^{\varepsilon}                       & \mbox{by } (\dag)             \\
		 & = & (\boxempty\Diamond)^{2+\varepsilon}(\Diamond\boxempty)^{k_0}                       & \mbox{by } \eq_{11}^\circ(k_0)\\
		& = & \boxempty(\Diamond\boxempty)^{k_0+\varepsilon}                                     &                               \\
	 & = & (\boxempty\Diamond)^{k_1}\boxempty                                                 & \mbox{by } (\S)        
\end{array}$$
In other words, $$\eq_{11}^\circ(k_0) \Rightarrow \eq_{10}^\circ(k_1).$$
Suppose now that $k_0>k_1$.
From $\eq_{10}^\circ(k_1)$, by Lemma \ref{lm-pair}, we deduce $\eq_{00}(k_1,2)$ which is equivalent to $(\boxempty\Diamond)^{k_1}=(\boxempty\Diamond)^{k_1+2}$ ($\dag$) by left and right multiplying by $\Diamond$ both sides of this last equation. As $k_1<k_0$, iterating $\eq_{00}(k_1,2)$, we also have $(\Diamond\boxempty)^{k_1}=(\Diamond\boxempty)^{k_0+\varepsilon}$ with $\varepsilon\in \{0,1\}$ ($\ddag$). Then $\eq_{01}^\circ(k_0)\wedge\eq_{10}^\circ(k_1)$ implies 
$$\begin{array}{rcll}
 (\Diamond\boxempty)^{k_1}& = & (\boxempty\Diamond)^{k_1}\boxempty                & \text{by }(\eq^{\circ}_{10}(k_1))\\
 & = & \boxempty(\Diamond\boxempty)^{k_0+\varepsilon}   & \text{by }(\ddag)\\
 & = & \boxempty(\Diamond\boxempty)^{k_0+\varepsilon+1}\Diamond & \text{by }(\eq^{\circ}_{01}(k_0))\\
 & = & \boxempty(\Diamond\boxempty)^{k_1+1}\Diamond & \text{by }(\ddag)\\
	& = & (\boxempty\Diamond)^{k_1}           & \text{by }(\dag)
\end{array}$$
In other words,
$$\eq^\circ_{01}(k_0)\wedge\eq_{10}^\circ(k_1)\Rightarrow \eq_{11}^\circ(k_1).$$
Conversely, right multiplying $\eq^{\circ}_{11}(k_1)$ by $\boxempty$, we obtain $\eq^{\circ}_{10}(k_1)$.\\
From $\eq^{\circ}_{11}(k_1)$, we obtain $\eq_{00}(k_1,2)$ by Lemma \ref{lm-pair}. Since $k_0>k_1$, $\eq_{00}(k_1,2)$ implies $\eq_{00}(k_0,2)$ and $(\Diamond\boxempty)^{k_0}=(\Diamond\boxempty)^{k_1+\varepsilon}$ ($\S$) for some $\varepsilon\in\{0,1\}$.
Then $\eq_{11}^\circ(k_1)$ implies
$$\begin{array}{rcll}
(\Diamond\boxempty)^{k_0} & = & (\Diamond\boxempty)^{k_1+2+\varepsilon} & \text{by } \eq_{00}(k_0,2) \text{ and } (\S)\\
& = & (\Diamond\boxempty)^{2+\varepsilon}(\Diamond\boxempty)^{k_1} & \\
& = & (\Diamond\boxempty)^{2+\varepsilon}(\boxempty\Diamond)^{k_1} & \text{by } \eq_{11}(k_1)\\
& = & (\Diamond\boxempty)^{1+k_1+\varepsilon}\Diamond & \\
& = & (\Diamond\boxempty)^{1+k_0}\Diamond & \text{by } (\S)\\
\end{array}$$
In other words,
$$ \eq_{11}^\circ(k_1)\Rightarrow\eq^\circ_{01}(k_0).$$
\end{proof}

\begin{lemma}\label{ptt-outil}
For every  $k, r > 0$, $i\in\{01,10\}$, we have 
$$\eq_{00}(k,2)\wedge\eq_i^\circ(k+r)\Rightarrow\eq_i^\circ(k).$$
\end{lemma}
\begin{proof}
We have to consider the two following cases:
\begin{itemize}
	\item If $r=2\ell$ is even, then we obtain  $\eq_{00}(k,r)$ by applying  $\ell$ times $\eq_{00}(k,2)$. Hence, $$\eq_i^\circ(k+r) \wedge \eq_{00}(k,r) \Rightarrow \eq_i^\circ(k).$$
	\item If $r=2\ell+1$ is odd, then we obtain $\eq_{00}(k,r+1)$ by applying  $\ell+1$ times $\eq_{00}(k,2)$. Left (resp. right)-multiplying both members of the equation $\eq_{01}^\circ(k+r)$ (resp. $\eq_{10}^\circ(k+r)$)  by $\Diamond\boxempty$, we find $\eq_{01}^\circ(k+r+1)$ (resp. $\eq_{10}^\circ(k+r+1)$). Finally, $$\eq_i^\circ(k+r+1) \wedge \eq_{00}(k,r+1) \Rightarrow \eq_i^\circ(k).$$
\end{itemize}
\end{proof}

\begin{lemma}\label{1et2<=>2}
For every $k_0, k_1, r > 0$, $i,j\in\{01,10,11\}$, with $i\neq j$, we have
\begin{equation}\label{eq1}
\eq_{00}(k_0,2r)\wedge\eq_i^\circ(k_1)\Leftrightarrow\eq_i^\circ(\min(k_0,k_1))
\end{equation}
\begin{equation}\label{eq2}
\eq_i^\circ(k_0)\wedge\eq_j^\circ(k_1)\Leftrightarrow\eq_{i|j}^\circ(\min(k_0,k_1))
\end{equation}
where $i|j$ denotes the bitwise disjunction of $i$ and $j$.
\end{lemma}
\begin{proof}
We first prove Equivalence (\ref{eq1}), for $i=01$ and $i=10$.
\begin{itemize}
\item Assume that $\eq_{00}(k_0,2r)\wedge\eq_i^\circ(k_1)$ occurs.
If $\min(k_0,k_1)=k_1$, the result is obvious. Otherwise, we have
$$ \eq_i^\circ(k_1) \Rightarrow \eq_{00}(k_1,2) \mbox{ (by  Lemma \ref{lm-pair})}$$
then
$$ \eq_{00}(k_1,2) \wedge \eq_{00}(k_0,2r) \Rightarrow \eq_{00}(k_0,2) \mbox{ (by Lemma \ref{lem:eq1+eq1=eq1})}$$
and finally
$$ \eq_{00}(k_0,2) \wedge \eq_i^\circ(k_1) \Rightarrow \eq_i^\circ(k_0) \mbox{ (by Lemma \ref{ptt-outil})}$$
\item Conversely, suppose that $\eq_i^\circ(\min(k_0,k_1))$ occurs.
If $\min(k_0,k_1)=k_0$, we have $\eq_i^\circ(k_0) \Rightarrow \eq_{00}(k_0,2)$ by Lemma \ref{lm-pair} 
then, by repeating $r$ times $\eq_{00}(k_0,2)$, we obtain $\eq_{00}(k_0,2r)$. Furthermore, left (resp. right)-multiplying $k_1-k_0$ times both members of the equation $\eq_{01}^\circ(k_0)$ (resp. $\eq_{10}^\circ(k_0)$) by $\Diamond\boxempty$, we obtain $\eq_{01}^\circ(k_1)$ (resp. $\eq_{10}^\circ(k_1)$).\\
Otherwise, we trivially obtain $\eq_i^\circ(k_1)$. Moreover, Lemma \ref{lm-pair}  gives us $\eq_{00}(k_1,2)$ and, by left (resp. right)-multiplying $k_0-k_1$ times both members of the equation $\eq_{01}^\circ(k_1)$ (resp. $\eq_{10}^\circ(k_1)$) by $\Diamond\boxempty$, we obtain $\eq_{01}^\circ(k_0)$ (resp. $\eq_{10}^\circ(k_0)$). Finally, we have $\eq_{00}(k_1,2)\wedge\eq_i^\circ(k_0)\Rightarrow\eq_i^\circ(k_1)$ by Lemma \ref{ptt-outil}.
\end{itemize}
We deduce Equivalence (\ref{eq1}) for $i=11$ because
$$\begin{array}{rlll}
\eq_{00}(k_0,2r)\wedge\eq_{11}^\circ(k_1) & \Leftrightarrow & \eq_{00}(k_0,2r)\wedge\eq_{01}^\circ(k_1)\wedge\eq_{10}^\circ(k_1) & (\mbox{Lemma }\ref{2et3<=>4})\\
 & \Leftrightarrow & \eq_{01}^\circ(\min(k_0,k_1))\wedge\eq_{10}^\circ(\min(k_0,k_1)) & \mbox{(previous equiv.)}\\
 & \Leftrightarrow & \eq_{11}^\circ(\min(k_0,k_1)) & (\mbox{Lemma } \ref{2et3<=>4})
\end{array}$$
Equivalence (\ref{eq2}) for $i=01$ and $j=10$ is shown in Lemma \ref{2et3<=>4}. Finally, we obtain  Equivalence (\ref{eq2}) when $j=11$ for $i=01$ or $i=10$ by
$$\begin{array}{rlll}
\eq_i^\circ(k_0)\wedge\eq_{11}^\circ(k_1) & \Leftrightarrow & \eq_i^\circ(k_0)\wedge\eq_{01}^\circ(k_1)\wedge\eq_{10}^\circ(k_1) & (\mbox{Lemma }\ref{2et3<=>4})\\
 & \Leftrightarrow & \eq_i^\circ(\min(k_0,k_1))\wedge\eq_{i'}^\circ(k_1) & \mbox{(Proposition \ref{k0k1<=>min})}\\
 & \Leftrightarrow & \eq_{11}^\circ(\min(k_0,k_1)) & (\mbox{Lemma } \ref{2et3<=>4})
\end{array}$$
where $i'=10$ if $i=01$ and $i'=01$ if $i=10$.
\end{proof}

\begin{lemma}\label{lm-eqi+eq1=eqi}
For every $i\in\{01,10,11\}$, we have $$\eq_i^\bullet(k)\Leftrightarrow\eq_i^\circ(k)\wedge\eq_{00}(k,1).$$
\end{lemma}
\begin{proof}
From $\eq_i^\bullet(k)$, we deduce $\eq_{00}(k,1)$ by Lemma \ref{lm-impair}. The three cases are detailed below:
\begin{itemize}
\item If $i=01$ then, by applying $\eq_{00}(k,1)$ to the right-hand side of $\eq_i^\bullet(k)$, we immediately obtain $\eq_i^\circ(k)$.
\item If  $i=10$ then $\eq_{10}^\bullet(k)$ is equivalent to
$$
\begin{array}{rcll}
 (\Diamond\boxempty)^k & = & (\boxempty\Diamond)^{k-1}\boxempty                 & \\
 & = & \Diamond\Diamond(\boxempty\Diamond)^{k-1}\boxempty & \\
& = & \Diamond(\Diamond\boxempty)^{k}                    & \\
 & = & \Diamond(\Diamond\boxempty)^{k+1}                  & \mbox{by }\eq_{00}(k,1)\\
 & = & \Diamond\Diamond(\boxempty\Diamond)^{k}\boxempty   & \\
 & = & (\boxempty\Diamond)^{k}\boxempty           
\end{array}
$$
Hence,
\[\eq_{10}^\bullet(k)\Leftrightarrow \eq_{10}^\circ(k).\]
\item If $i=11$ then $\eq_{11}^\bullet(k)$ is equivalent to
$$
\begin{array}{rcll}
 (\Diamond\boxempty)^k & = & (\boxempty\Diamond)^{k+1}                          & \\
 & = & \Diamond\Diamond(\boxempty\Diamond)^{k+1}          & \\
 & = & \Diamond(\Diamond\boxempty)^{k+1}\Diamond          & \\
& = & \Diamond(\Diamond\boxempty)^{k}\Diamond            & \mbox{by }\eq_{00}(k,1)\\
 & = & \Diamond\Diamond(\boxempty\Diamond)^{k}            & \\
 = & (\boxempty\Diamond)^{k}                            & \Leftrightarrow  \eq_{11}^\circ(k).
\end{array}
$$
Hence,
\[\eq_{11}^\bullet(k)\Leftrightarrow \eq_{11}^\circ(k).\]
\end{itemize}
\end{proof}

\begin{lemma}\label{1et2bullet<=>2}
For every $k_0, k_1, r > 0$, $i,j\in\{01,10,11\}$, with $i\neq j$, we have
\begin{equation}\label{eq1b}
\eq_{00}(k_0,2r+1)\wedge\eq_i^\bullet(k_1)\Leftrightarrow\eq_i^\bullet(\min(k_0,k_1))
\end{equation}
\begin{equation}\label{eq2b}
\eq_i^\bullet(k_0)\wedge\eq_j^\bullet(k_1)\Leftrightarrow\eq_{i|j}^\bullet(\min(k_0,k_1))
\end{equation}
\end{lemma}
\begin{proof}
We obtain (\ref{eq1b}) by applying the following sequence of equivalences:
$$\begin{array}{rll}
                & \eq_{00}(k_0,2r+1)\wedge\eq_i^\bullet(k_1) & \\
\Leftrightarrow & \eq_{00}(k_0,2(2r+1))\wedge\eq_i^\bullet(k_1) & (\mbox{iterating }\eq_{00}(k_0,2r+1))\\
\Leftrightarrow & \eq_{00}(k_0,2(2r+1))\wedge\eq_i^\circ(k_1)\wedge\eq_{00}(k_1,1) & \mbox{(Lemma \ref{lm-eqi+eq1=eqi})}\\
\Leftrightarrow & \eq_i^\circ(\min(k_0,k_1))\wedge\eq_{00}(\min(k_0,k_1),1) & \mbox{(Lemmas \ref{1et2<=>2}\mbox{ et }\ref{lem:eq1+eq1=eq1})}\\
\Leftrightarrow & \eq_i^\bullet(\min(k_0,k_1)) & \mbox{(Lemma \ref{lm-eqi+eq1=eqi})}
\end{array}$$
We obtain (\ref{eq2b}) by applying the following sequence of equivalences:
$$\begin{array}{rll}
                & \eq_i^\bullet(k_0)\wedge\eq_j^\bullet(k_1) & \\
\Leftrightarrow & \eq_i^\circ(k_0)\wedge\eq_{00}(k_0,1)\wedge\eq_j^\circ(k_1)\wedge\eq_{00}(k_1,1) & \mbox{(Lemma \ref{lm-eqi+eq1=eqi})}\\
\Leftrightarrow & \eq_{i|j}^\circ(\min(k_0,k_1))\wedge\eq_{00}(\min(k_0,k_1),1) & \mbox{(Lemmas \ref{1et2<=>2}\mbox{ and }\ref{lem:eq1+eq1=eq1})}\\
\Leftrightarrow & \eq_{i|j}^\bullet(\min(k_0,k_1)) & \mbox{(Lemma \ref{lm-eqi+eq1=eqi})}
\end{array}$$
\end{proof}

We have considered all combinations of equations belonging to sets suffixed with either $\circ$ or $\bullet$.
What remains to be studied are combinations.

\begin{lemma}\label{melange}
For every $k_0, k_1, r > 0$, $i,j\in\{01,10,11\}$, with $i\neq j$, we have
\begin{equation}\label{eq1m}
\eq_i^\bullet(k_0)\wedge\eq_{00}(k_1,2r)\Leftrightarrow\eq_i^\bullet(\min(k_0,k_1))
\end{equation}
\begin{equation}\label{eq2m}
\eq_{00}(k_0,2r+1)\wedge\eq_i^\circ(k_1)\Leftrightarrow\eq_i^\bullet(\min(k_0,k_1))
\end{equation}
\begin{equation}\label{eq3m}
\eq_i^\circ(k_0)\wedge\eq_j^\bullet(k_1)\Leftrightarrow\eq_{i|j}^\bullet(\min(k_0,k_1))
\end{equation}
\end{lemma}
\begin{proof} Let us first show equality (\ref{eq1m}).
$$\begin{array}{rll}
                & \eq_i^\bullet(k_0)\wedge\eq_{00}(k_1,2r) & \\
\Leftrightarrow & \eq_i^\circ(k_0)\wedge\eq_{00}(k_0,1)\wedge\eq_{00}(k_1,2r) & \mbox{(Lemma \ref{lm-eqi+eq1=eqi})}\\
\Leftrightarrow & \eq_i^\circ(\min(k_0,k_1))\wedge\eq_{00}(k_0,1)\wedge\eq_{00}(k_1,2r) & \mbox{(Lemma \ref{1et2<=>2})}\\
\Leftrightarrow & \eq_i^\circ(\min(k_0,k_1))\wedge\eq_{00}(\min(k_0,k_1),1) & \mbox{(Lemma \ref{lem:eq1+eq1=eq1})}\\
\Leftrightarrow & \eq_i^\bullet(\min(k_0,k_1)) & \mbox{(Lemma \ref{lm-eqi+eq1=eqi})}
\end{array}$$

For equality (\ref{eq2m}), we have
$$\begin{array}{rll}
                & \eq_{00}(k_0,2r+1)\wedge\eq_i^\circ(k_1) & \\
\Leftrightarrow & \eq_{00}(k_0,2r+1)\wedge\eq_{00}(k_0,2(2r+1))\wedge\eq_i^\circ(k_1) & (\text{iterating }\eq_{00}(k_0,2r+1))\\
\Leftrightarrow & \eq_{00}(k_0,2r+1)\wedge\eq_i^\circ(\min(k_0,k_1)) & \mbox{(Lemma \ref{1et2<=>2})}\\
\Leftrightarrow & \eq_{00}(k_0,2r+1)\wedge\eq_{00}(\min(k_0,k_1),2)\wedge\eq_i^\circ(\min(k_0,k_1)) & \mbox{(Lemma \ref{lm-pair})}\\
\Leftrightarrow & \eq_{00}(\min(k_0,k_1),1)\wedge\eq_i^\circ(\min(k_0,k_1)) & \mbox{(Lemma \ref{lem:eq1+eq1=eq1})}\\
\Leftrightarrow & \eq_i^\bullet(\min(k_0,k_1)) & \mbox{(Lemma \ref{lm-eqi+eq1=eqi})}
\end{array}$$

And for equality (\ref{eq3m}), we have
$$\begin{array}{rll}
                & \eq_i^\circ(k_0)\wedge\eq_j^\bullet(k_1) & \\
\Leftrightarrow & \eq_i^\circ(k_0)\wedge\eq_{00}(k_0,2)\wedge\eq_j^\circ(k_1)\wedge\eq_{00}(k_1,1) & (\mbox{Lemmas \ref{lm-pair} et \ref{lm-eqi+eq1=eqi}})\\
\Leftrightarrow & \eq_{i|j}^\circ(\min(k_0,k_1))\wedge\eq_{00}(\min(k_0,k_1),1) & (\mbox{Lemmas \ref{1et2<=>2} et \ref{lem:eq1+eq1=eq1}})\\
\Leftrightarrow & \eq_{i|j}^\bullet(\min(k_0,k_1)) & \mbox{(Lemma \ref{lm-eqi+eq1=eqi})}
\end{array}$$
\end{proof}

Actually, the set of classes $\E_i^s$, $i\in\{00,01,10,11\}$, $s\in\{\circ,\bullet\}$ constitutes a Boolean algebra if we provide it with the following order relation: $E<F$ if and only if for every $f\in F$, there exists $e\in E$ such that $e \Rightarrow f$. The structure of this algebra can be modeled by the lattice of Figure \ref{latticeEqu}.

\begin{figure}[ht]
\centerline{
\begin{tikzpicture}[node distance=1.5cm, bend angle=15]
\node(q1c)[state]{$\E_{00}^{\circ}$};
\node(q1b)[state,below of = q1c]{$\E_{00}^{\bullet}$};
\node(q2c)[state,left of = q1b]{$\E_{01}^{\circ}$};
\node(q3c)[state,right of = q1b]{$\E_{10}^{\circ}$};
\node(q4c)[state,below of = q1b]{$\E_{11}^{\circ}$};
\node(q2b)[state,left of = q4c]{$\E_{01}^{\bullet}$};
\node(q3b)[state,right of = q4c]{$\E_{10}^{\bullet}$};
\node(q4b)[state,below of = q4c]{$\E_{11}^{\bullet}$};
\path[->]
(q4b)edge[]node[]{}(q2b)
(q4b)edge[]node[]{}(q4c)
(q4b)edge[]node[]{}(q3b)
(q2b)edge[]node[]{}(q2c)
(q2b)edge[]node[]{}(q1b)
(q4c)edge[]node[]{}(q2c)
(q4c)edge[]node[]{}(q3c)
(q3b)edge[]node[]{}(q3c)
(q3b)edge[]node[]{}(q1b)
(q2c)edge[]node[]{}(q1c)
(q1b)edge[]node[]{}(q1c)
(q3c)edge[]node[]{}(q1c);
\end{tikzpicture}}
\caption{Lattice of the $\E^s_i$ classes}\label{latticeEqu}
\end{figure}
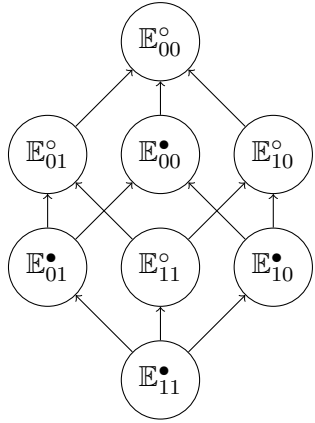

In this lattice, every pair of elements $E, F$ has a lower bound, denoted $E\wedge F$. The above results show that every pair of equations belonging to $E\times F$ has an equivalent equation belonging to $E\wedge F$.

\subsection{Theorem of reduction}
To summarize the results of this section, any presentation of a strict 2-PIM monoid reduces to a presentation with only one additional equation.
\begin{theorem}\label{ThRed}
Any strict 2-PIM has a presentation reduced to three equations (including $\Diamond^2=\Id $ and $\boxempty^2=\boxempty$).
\end{theorem}

\begin{proof}
The result follows from  Propositions \ref{E00XE00<=>E00}, \ref{k0k1<=>min} and  Lemmas \ref{1et2<=>2}, \ref{1et2bullet<=>2} and \ref{melange}.
\end{proof}

We summarize in Table \ref{tab-resume} the classes of strict 2-PIMs  with their unique relation if it exists. All these cases are non-degenerate for positive $k$ and $r$.
\begin{table}[h]
$$\begin{array}{l|l}
\text{Monoid}&\text{Additional relation}\\
\hline
M_\emptyset&\text{None}\\
\hline
M_{00}(k,r)&(\Diamond\boxempty)^k=(\Diamond\boxempty)^{k+r}\\
\hline
M_{01}^\circ(k)&(\Diamond\boxempty)^k=(\Diamond\boxempty)^{k+1}\Diamond\\
\hline
M_{01}^\bullet(k)&(\Diamond\boxempty)^k=(\Diamond\boxempty)^{k}\Diamond\\
\hline
M_{10}^\circ(k)&(\Diamond\boxempty)^k=(\boxempty\Diamond)^{k}\boxempty\\
\hline
M_{10}^\bullet(k)&(\Diamond\boxempty)^k=(\boxempty\Diamond)^{k-1}\boxempty\\
\hline
M_{11}^\circ(k)&(\Diamond\boxempty)^k=(\boxempty\Diamond)^{k}\\
\hline
M_{11}^\bullet(k)&(\Diamond\boxempty)^k=(\boxempty\Diamond)^{k+1}\\
\end{array}$$
\caption{Classes of strict $2$-PIMs}\label{tab-resume}
\end{table}

This theorem is not strictly speaking a classification result. It remains to be proved that the cases identified are indeed distinct.

\section{Lattice structure on sets of strict 2-PIMs}

We set, for all $i\in\{00,01,10,11\}$ and $s\in \{\circ,\bullet,\varepsilon\}$
\[
\M_i^s=\{\monoide{eq}\mid\mathrm{eq}\in\E_i^s, \mathrm{E}\subset\Eq\}
\]
In other words, $\M_i^s$ (resp. $\M_i$ if $s=\varepsilon$) is the set of monoids for which there exists a presentation containing a $\E_i^s$  (resp. $\E_i$) equation. 
\subsection{An inclusion diagram}
The following lemma allows us to show the inclusions of the lattice of Figure \ref{latticeMonoide}, while  Lemma \ref{lm-eqimpIneqpair} allows us to show that these inclusions are strict.
\begin{lemma}\label{lm-outil-MiSubsetMj}
Let $i,j\in\{00,01,10,11\}$ and $s,s'\in\{\circ,\bullet,\varepsilon\}$. If for any equation $\mathrm{eq}$ of $\E_i^s$ there exists an equation $\mathrm{eq}'$ of $\E_j^{s'}$, such that $\mathrm{eq}\Rightarrow \mathrm{eq}'$
then $\M_i^s\subseteq \M_j^{s'}$.
\end{lemma}
\begin{proof}
Let $\mathrm{M}=\monoide{eq}\in \M_i^s$ then $$\mathrm{M}=\monoide{eq, eq'}=\monoide[E\cup \{eq\}]{eq'}\in \M_j^{s'}.$$
\end{proof}

Applying Lemma \ref{lm-outil-MiSubsetMj}, to the lattice of Figure \ref{latticeEqu}, we obtain the  inclusion diagram of Figure \ref{latticeMonoide}.

\begin{figure}[ht]
\centerline{
\begin{tikzpicture}[node distance=2cm, bend angle=15]
\node(q1c)[state]{$\M_{00}^{\circ}$};
\node(q1b)[state,below of = q1c]{$\M_{00}^{\bullet}$};
\node(q2c)[state,left of = q1b]{$\M_{01}^{\circ}$};
\node(q3c)[state,right of = q1b]{$\M_{10}^{\circ}$};
\node(q4c)[state,below of = q1b]{$\M_{11}^{\circ}$};
\node(q2b)[state,left of = q4c]{$\M_{01}^{\bullet}$};
\node(q3b)[state,right of = q4c]{$\M_{10}^{\bullet}$};
\node(q4b)[state,below of = q4c]{$\M_{11}^{\bullet}$};
\path[-]
(q4b)edge[below]node[sloped]{$\supseteq$}(q2b)
(q4b)edge[below]node[sloped]{$\subseteq$}(q4c)
(q4b)edge[below]node[sloped]{$\subseteq$}(q3b)
(q2b)edge[above]node[sloped]{$\subseteq$}(q2c)
(q2b)edge[below]node[pos=.15,sloped]{$\subseteq$}(q1b)
(q4c)edge[above]node[pos=.85,sloped]{$\supseteq$}(q2c)
(q4c)edge[below]node[pos=.15,sloped]{$\subseteq$}(q3c)
(q3b)edge[below]node[sloped]{$\subseteq$}(q3c)
(q3b)edge[above]node[pos=.85,sloped]{$\supseteq$}(q1b)
(q2c)edge[above]node[sloped]{$\subseteq$}(q1c)
(q1b)edge[below]node[sloped]{$\subseteq$}(q1c)
(q3c)edge[above]node[sloped]{$\supseteq$}(q1c);
\end{tikzpicture}}
\caption{Inclusion lattice of the $\M^s_i$ sets}\label{latticeMonoide}
\end{figure}
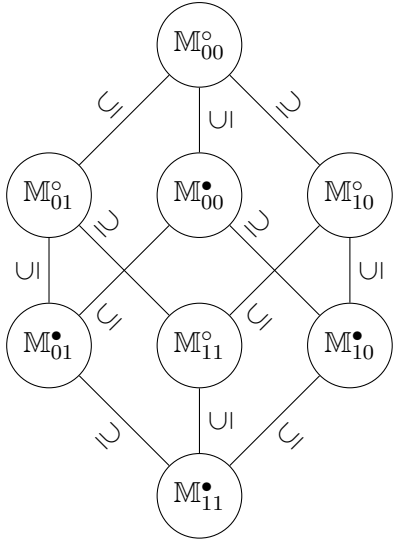

%

%

\begin{lemma}\label{lm-eqimpIneqpair}
We have
\begin{eqnarray}
\M_i^\bullet &\neq & \M_i^{\circ}\text{ for }i\in \{00,01,10,11\}\label{ineq1}\\
\M_{11}^s&\neq& \M_{01}^s\text{ for }s\in\{\circ,\bullet\}\label{ineq2}\\
\M_{11}^s&\neq& \M_{10}^s\text{ for }s\in\{\circ,\bullet\}\label{ineq3}\\
\M_{00}^s&\neq& \M_{01}^s\text{ for }s\in\{\circ,\bullet\}\label{ineq4}\\
\M_{00}^s&\neq& \M_{10}^s\text{ for }s\in\{\circ,\bullet\}\label{ineq5}
\end{eqnarray}
\end{lemma}

\begin{proof}
%
%

To show (\ref{ineq1}), we need only exhibit a monoid of $\M_i^{\circ}\backslash\M_i^\bullet$ for $i\in \{00,01,10,11\}$. To do this, consider the set of matrices $2\times2$ with integer coefficients, equipped with the matrix product. 
\begin{enumerate}
\item[case $i=00$]
If we set
$$\Diamond=\begin{pmatrix}
-1 & y \\
0 & 1
\end{pmatrix}
\mathrm{\ and\ }
\boxempty=\begin{pmatrix}
1 & 0 \\
0 & 0
\end{pmatrix}$$
We can check that $(\Diamond\boxempty)^k=\begin{pmatrix}(-1)^k&0 \\0&0\end{pmatrix}$. We therefore have a value for $k$, such that $(\Diamond\boxempty)^{k}=(\Diamond\boxempty)^{k+2\ell}$ but $(\Diamond\boxempty)^{k}\neq(\Diamond\boxempty)^{k+2\ell+1}$. The monoid under consideration therefore belongs to $\M_{00}^{\circ}$ and does not belong to $\M_{00}^\bullet$.

\item[case $i=01$]
If we set
$$\Diamond=\begin{pmatrix}
1 & y \\
0 & -1
\end{pmatrix}
\mathrm{\ and\ }
\boxempty=\begin{pmatrix}
0 & 0 \\
0 & 1
\end{pmatrix}$$
We can check that $\Diamond\boxempty\Diamond\boxempty\Diamond=\Diamond\boxempty\neq\Diamond\boxempty\Diamond$. We therefore have a value for $k$ (in this case $1$), such that $(\Diamond\boxempty)^{k}=(\Diamond\boxempty)^{k+1}\Diamond$ but $(\Diamond\boxempty)^{k}\neq(\Diamond\boxempty)^{k}\Diamond$. 

It is easy to check that
$$\forall\ell>0, (\Diamond\boxempty)^{2\ell}=\begin{pmatrix}0 & -y \\0 & 1\end{pmatrix}\neq(\Diamond\boxempty)^{2\ell}\Diamond=\begin{pmatrix}0 & y \\0 & -1\end{pmatrix}$$
and
$$\forall\ell\geq0, (\Diamond\boxempty)^{2\ell+1}=\begin{pmatrix}0 & y \\0 & -1\end{pmatrix}\neq(\Diamond\boxempty)^{2\ell+1}\Diamond=\begin{pmatrix}0 & -y \\0 & 1\end{pmatrix}.$$
The monoid under consideration therefore belongs to $\M_{01}^{\circ}$ and does not belong to $\M_{01}^\bullet$.
\item[case $i=10$]
\[
\Diamond=\begin{pmatrix}-1&y\\0&1\end{pmatrix} \mbox{ and }\boxempty=\begin{pmatrix}1&0\\0&0\end{pmatrix}.
\]

We have for any $k>0$:
\[
(\Diamond\boxempty)^k=(\boxempty\Diamond)^k\boxempty=\begin{pmatrix}(-1)^k&0 \\0&0\end{pmatrix}
\]
which also proves that for any $k>1$, $(\Diamond\boxempty)^k\neq (\boxempty\Diamond)^{k-1}\boxempty$ and thus $\M_{10}^\bullet\neq\M_{10}^{\circ}$.

\item[case $i=11$]
\[
\Diamond=\begin{pmatrix}-1&0\\0&1\end{pmatrix} \mbox{ and }\boxempty=\begin{pmatrix}1&0\\0&0\end{pmatrix}.
\]

We have for any $k>0$:
\[
(\Diamond\boxempty)^k=(\boxempty\Diamond)^k=\begin{pmatrix}(-1)^k&0 \\0&0\end{pmatrix}
\]
which also proves that for any $k>1$, $(\Diamond\boxempty)^k\neq (\boxempty\Diamond)^{k-1}\boxempty$ and thus $\M_{11}^{\bullet}\neq\M_{11}^{\circ}$.

\end{enumerate}
To show (\ref{ineq5}), we set
$$\Diamond=\begin{pmatrix}
1 & 0 \\
y & -1
\end{pmatrix}
\mathrm{\ and\ }
\boxempty=\begin{pmatrix}
1 & 0 \\
0 & 0
\end{pmatrix}.$$
For all $k>0$, we have $(\Diamond\boxempty)^k=(\Diamond\boxempty)^{k+1}=(\Diamond\boxempty)^{k+2}=\begin{pmatrix}
1 & 0 \\
y & 0
\end{pmatrix}$ while $(\boxempty\Diamond)^{k}\boxempty=(\boxempty\Diamond)^{k-1}\boxempty= \begin{pmatrix}
1 & 0 \\
0 & 0
\end{pmatrix}$.\\

With these  matrices, we also have for all $k>0$, $(\Diamond\boxempty)^k=(\Diamond\boxempty)^k\Diamond=(\Diamond\boxempty)^{k+1}\Diamond$ while $(\boxempty\Diamond)^k=(\boxempty\Diamond)^{k+1}\neq(\Diamond\boxempty)^k$ which proves (\ref{ineq2}).

To show (\ref{ineq4}), we set
$$\Diamond=\begin{pmatrix}
1 & 0 \\
y & 1
\end{pmatrix}
\mathrm{\ and\ }
\boxempty=\begin{pmatrix}
0 & 0 \\
0 & 1
\end{pmatrix}.$$
For all $k>0$, we have $(\Diamond\boxempty)^k=(\Diamond\boxempty)^{k+1}=(\Diamond\boxempty)^{k+2}=\begin{pmatrix}
0 & 0 \\
0 & 1
\end{pmatrix}$ while $(\Diamond\boxempty)^{k}\Diamond=(\Diamond\boxempty)^{k+1}\Diamond= \begin{pmatrix}
0 & 0 \\
y & 1
\end{pmatrix}$.\\

To show (\ref{ineq3}), we set
$$\Diamond=\begin{pmatrix}
-1 & 0 \\
y & 1
\end{pmatrix}
\mathrm{\ and\ }
\boxempty=\begin{pmatrix}
0 & 0 \\
0 & 1
\end{pmatrix}.$$
For all $k>0$, we have $(\Diamond\boxempty)^k=(\boxempty\Diamond)^{k-1}\boxempty=(\boxempty\Diamond)^{k}\boxempty=\begin{pmatrix}
0 & 0 \\
0 & 1
\end{pmatrix}$ while $(\boxempty\Diamond)^{k}= (\boxempty\Diamond)^{k+1}= \begin{pmatrix}
0 & 0 \\
y & 1
\end{pmatrix}$.

\end{proof}

\subsection{Intersections of the sets $\M_i$}
We now show what the intersections of the different sets of monoids correspond to.

\begin{lemma}\label{lm-M4=M3capM2}
$$\M_{11}=\M_{10}\cap \M_{01}.$$
\end{lemma}
\begin{proof}
Let us show this result by double inclusion.
The result of Lemma \ref{lm-outil-MiSubsetMj} gives us one direction of inclusion which is $\M_{11}\subset \M_{01}\cap \M_{10}$. For the other direction, let's assume we have the following equations:
\begin{equation}\label{eqM3}
(\Diamond\boxempty)^{k_0}=\boxempty(\Diamond\boxempty)^{k_0} \hfill \in \E_{10}
\end{equation}
and
\begin{equation}\label{eqM2}
(\Diamond\boxempty)^{k_1}=(\Diamond\boxempty)^{k_1+1}\Diamond \hfill \in \E_{01}
\end{equation}
If $k_0\leq k_1$, Equation (\ref{eqM2}) can be rewritten as follows:
$$\begin{array}{lll}
(\Diamond\boxempty)^{k_1}&=&\Diamond\boxempty(\Diamond\boxempty)^{k_0}(\Diamond\boxempty)^{k_1-k_0}\Diamond\\
&=&\Diamond(\Diamond\boxempty)^{k_0}(\Diamond\boxempty)^{k_1-k_0}\Diamond\\
&=&(\boxempty\Diamond)^{k_1}
\end{array}$$
This gives the equation $(\Diamond\boxempty)^{k_1}=(\boxempty\Diamond)^{k_1}$. \\

Suppose now that $k_1<k_0$. By left and right multiplying by $\Diamond$, Equation (\ref{eqM2}) becomes
$$(\boxempty\Diamond)^{k_1}=(\boxempty\Diamond)^{k_1}\boxempty$$
We can now rewrite Equation (\ref{eqM3}).
$$\begin{array}{lll}
(\Diamond\boxempty)^{k_0}&=&(\boxempty\Diamond)^{k_0-k_1}(\boxempty\Diamond)^{k_1}\boxempty\\
&=&(\boxempty\Diamond)^{k_0-k_1}(\boxempty\Diamond)^{k_1}\\
&=&(\boxempty\Diamond)^{k_0}
\end{array}$$
This gives the equation $(\Diamond\boxempty)^{k_1}=(\boxempty\Diamond)^{k_1}$.  In both cases, we obtain an equation of $\E_{11}$. We use Lemma \ref{lm-outil-MiSubsetMj}  to conclude the proof.
\end{proof}

\begin{lemma}
For every $i\in\{01,10,11\}$, we have $\M_i^\bullet=\M_i\cap \M_{00}^\bullet$.
\end{lemma}
\begin{proof}
By Lemma \ref{lm-eqimpIneqpair}, we have $\M_i^\bullet \subset \M_i$. By Lemmas \ref{lm-impair} and \ref{lm-outil-MiSubsetMj} we can conclude that $\M_i^\bullet \subset \M_{00}^\bullet$. For the inverse inclusion, the case $\mathrm{M}\in \M_i^\bullet \cap \M_{00}^\bullet$ is trivial. Let us consider the case where $\mathrm{M}\in \M_i^\circ\cap \M_{00}^\bullet$. We then have that $\mathrm{M}$ is presented by an equation of $\E_i^\circ$ and an equation of $\E_{00}^\bullet$. We then conclude by Lemmas \ref{lm-eqi+eq1=eqi} and \ref{lm-outil-MiSubsetMj}.
\end{proof}
 
\bigskip
\begin{figure}[h]
 \centerline{
\begin{tikzpicture}
\draw (0,0) node[rectangle,rounded corners,minimum height=5.5cm,minimum width=8cm,draw] {};
\draw (-.5,.5) node[rectangle,rounded corners,minimum height=3cm,minimum width=5cm,fill=red!5] {};
\draw (.5,-.5) node[rectangle,rounded corners,minimum height=3cm,minimum width=5cm,fill=blue!5, opacity=0.5] {};
\draw (-.5,.5) node[rectangle,rounded corners,minimum height=3cm,minimum width=5cm,draw=red] {};
\draw (.5,-.5) node[rectangle,rounded corners,minimum height=3cm,minimum width=5cm,draw=blue] {};
\draw (0,-.5) node[minimum height=3cm,minimum width=3cm,draw, pattern=north west lines, pattern color=blue,draw=blue] {};
\draw (0,0.5) node[minimum height=3cm,minimum width=3cm, pattern=north east lines, pattern color=red] {};
\draw (0,.5) node[minimum height=3cm,minimum width=3cm,draw,draw=red,opacity=0.8] {};
\draw (0,0) node[rectangle,rounded corners,minimum height=5cm,minimum width=3cm,draw=black] {};
\draw (5,2.2) node  {$\M_{00}^\bullet$};
\draw (4.7,2.2) -- (0.6,2.2);
\draw (5,1) node  {$\M_{00}$};
\draw (4.7,1) -- (3.6,1);
\draw[color=red] (-5,-0.8) node  {$\M_{01}$};
\draw[color=red] (-4.7,-0.8) -- (-2.8,-0.8);
\draw[color=green] (-5,0) node  {$\M_{11}$};
\draw[color=green] (-4.7,0) -- (-1.8,0);
\draw[color=green] (-5,0.8) node  {$\M_{11}^\bullet$};
\draw[color=green] (-4.7,0.8) -- (-1.1,0.8);
\draw[color=red] (-5,1.7) node  {$\M_{01}^\bullet$};
\draw[color=red] (-4.7,1.7) -- (-1,1.7);
\draw[color=blue] (5,0) node  {$\M_{10}$};
\draw[color=blue] (4.7,0) -- (2.8,0);
\draw[color=blue] (5,-1.25) node  {$\M_{10}^{\bullet}$};
\draw[color=blue] (4.6,-1.25) -- (1,-1.25);
\end{tikzpicture}}
   \caption{Summary of the situation}
    \label{Fig-Summary}
\end{figure}
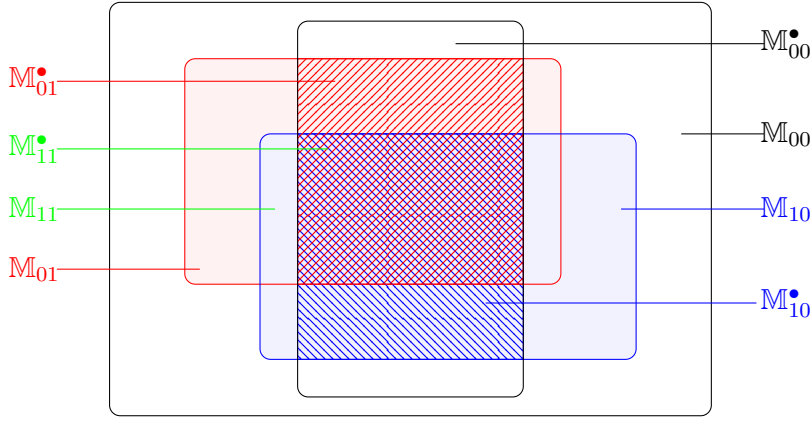

\section{Size of strict 2-PIMs}\label{sect-size}

To prove that our result gives a classification, we first need to count the elements of each monoid. Let us start with the case of  monoids characterized by an equation $\eq_{00}(k,k+r)$ ($(\Diamond\Box)^k=(\Diamond\Box)^{k+r}$).
There are $4$ possibilities for the minimal writing of an element:

\begin{itemize}
\item $(\Diamond\Box)^i$ for $i\in [0, k+r-1]$.
\item $(\Diamond\Box)^i\Diamond $ for $i\in[0, k+r-1]$.
\item $(\Box\Diamond)^i$ for $i\in[1, k+r-1]$. Be careful not to count twice the empty word ($i=0$) already counted in case 1.
\item $(\Box\Diamond)^i\Box$ for $i\in [0, k+r-2]$. Indeed, 
$(\Box\Diamond)^{k+r-1}\Box=\Diamond\Diamond(\Box\Diamond)^{k+r-1}\Box=\Diamond(\Diamond\Box)^{k+r}=\Diamond(\Diamond\Box)^{k}=(\Box\Diamond)^{k-1}\Box$.
\end{itemize}
This implies that the cardinality of the monoid is  $2(k+r)+2(k+r-1)=2k+2r-2$ elements.\\
For instance, we recover the case of the Kuratowski monoid, which is given by the equation $(\Diamond\Box)^4=(\Diamond\Box)^2$ ($eq_{00}(2,2)$). We find its cardinal $4*4-2=14$.\\
We can refine this enumeration, looking at the length of the smallest word $lw(e)$ representing element $e$ and counting the number of elements $e$ for a fixed $lw(e)$.
This information is summarized in the Hilbert series
\[
\mathcal H_{M}(t)=\sum_{e\in M}t^{lw(e)}.
\]
For example, for $M_{00}(k,r)$ the Hilbert series is easily deduced from the enumeration above. We obtain
\[
\mathcal H_{M_{00}(k,r)}(t)=1+\left(\sum_{i=1}^{2(k+r-1)}2t^i\right)+ t^{2(k+r-1)+1}.
\] 
For the Kuratowski monoid, we obtain
\[
\mathcal H_{M_{00}(2,2)}(t)=1+2(t+t^2+t^3+t^4+t^5+t^6)+ t^{7}.
\] 
This corresponds to the following graduation on the elements
\[
\begin{array}{|c|c|c|c|c|c|c|c|}
\hline
0&1&2&3&4&5&6&7\\\hline
\varepsilon&\Diamond&\Diamond\Box&\Diamond\Box\Diamond&(\Diamond\Box)^2&(\Diamond\Box)^2\Diamond&(\Diamond\Box)^3&(\Diamond\Box)^3\Diamond\\
&\Box&\Box\Diamond&\Box\Diamond\Box&(\Box\Diamond)^2&(\Box\Diamond)^2\Box&(\Box\Diamond)^3&\\\hline
\end{array}
\]
It is possible to systematize the calculation. First, let us consider the simplest case of the monoid $M_\emptyset$.
When a relation holds, it is the highest degree member that contains the element to be removed from the series. Applying  inclusion-exclusion principle, we obtain the series
\begin{equation}\label{H_0}
\mathcal H_{M_\emptyset}(t)=\frac1{1-2t+2t^2-2t^3+\cdots}=\frac1{1-\frac{2t}{1+t}}=\frac{1+t}{1-t}=1+2t+2t^2+\cdots.
\end{equation}
When a relation between the generators is added, we need to consider the equivalent one that has the smallest degree (the highest degree of the two words in the equation) and remove the Hilbert series
 of the ideal generated by the larger of the two words. For instance in $M_{00}(k,r)$, the equation $(\Diamond\Box)^k=(\Diamond\Box)^{k+r}$ is not the smallest one, because  by left multiplying by $\Diamond$, we obtain $\Box(\Diamond\Box)^{k-1}=\Box(\Diamond\Box)^{k+r-1}$. So we remove from $\mathcal H_{M_\emptyset}(t)$ the Hilbert series of the ideal generated by $\Box(\Diamond\Box)^{k+r-1}$. That is
\[\begin{array}{rcl}
\mathcal H_{M_{00}(k,r)}(t)&=&\mathcal H_{M_\emptyset}(t)-t^{2(k+r)-1}\mathcal H_{M_\emptyset}(t)=\frac{(1-t^{2(k+r)-1})(1+t)}{1-t}\\
&=&(1+t+\cdots +t^{2(k+r)-2})(1+t)
=1+2(t+\cdots+t^{2(k+r-1)})+t^{2(k+r-1)+1},
\end{array}
\]
as expected.\\
Now we examine the remaining cases by applying the same method
\begin{itemize}
\item For $M^\circ_{01}(k)$, we consider the relation $(\Box\Diamond)^{k}=(\Box\Diamond)^{k}\Box$ obtained by
conjugating $(\Diamond\Box)^k=(\Diamond\Box)^{k+1}\Diamond$ by $\Diamond$ and obtain
\[
\mathcal H_{M^\circ_{01}(k)}(t)=\frac{(1-t^{2k+1})(1+t)}{1-t}=1+2(t+t^2+\cdots+t^{2k})+t^{2k+1}.
\]
This yields $\#M^\circ_{01}(k)=\mathcal H_{M^\circ_{01}(k)}(1)=4k+2$.
\item For $M^\bullet_{01}(k)$, we consider the relation $(\Box\Diamond)^{k}=(\Box\Diamond)^{k-1}\Box$ obtained by
conjugating $(\Diamond\Box)^k=(\Diamond\Box)^{k}\Diamond$ by $\Diamond$ and obtain
\[
\mathcal H_{M^\bullet_{01}(k)}(t)=\frac{(1-t^{2k})(1+t)}{1-t}=1+2(t+t^2+\cdots+t^{2k-1})+t^{2k}.
\]
This yields $\#M^\bullet_{01}(k)=\mathcal H_{M^\bullet_{01}(k)}(1)=4k$.
\item For $M^\circ_{10}(k)$, we consider the relation $(\Box\Diamond)^{k}=(\Box\Diamond)^{k}\Box$ and obtain
\[
\mathcal H_{M^\circ_{10}(k)}(t)=\frac{(1-t^{2k+1})(1+t)}{1-t}=1+2(t+t^2+\cdots+t^{2k})+t^{2k+1}.
\]
This yields $\#M^\circ_{10}(k)=\mathcal H_{M^\circ_{10}(k)}(1)=4k+2$.
\item For $M^\bullet_{10}(k)$, we consider the relation $(\Box\Diamond)^{k}=(\Box\Diamond)^{k-1}\Box$ and obtain
\[
\mathcal H_{M^\bullet_{10}(k)}(t)=\frac{(1-t^{2k})(1+t)}{1-t}=1+2(t+t^2+\cdots+t^{2k-1})+t^{2k}.
\]
This yields $\#M^\bullet_{10}(k)=\mathcal H_{M^\circ_{10}(k)}(1)=4k$.
 \item For $M^\circ_{11}(k)$, we consider the relation $(\Box\Diamond)^{k}=(\Diamond\Box)^{k}$ and obtain
\[
\mathcal H_{M^\circ_{11}(k)}(t)=\frac{(1-t^{2k})(1+t)}{1-t}=1+2(t+t^2+\cdots+t^{2k-1})+t^{2k}.
\]
This yields $\#M^\circ_{11}(k)=\mathcal H_{M^\circ_{11}(k)}(1)=4k$.
 \item For $M^\bullet_{11}(k)$, we consider the relation $(\Box\Diamond)^{k}\Box=(\Diamond\Box)^{k}\Diamond$ obtained by right multiplying
 $(\Box\Diamond)^{k+1}=(\Diamond\Box)^{k}$ by $\Diamond$
 and obtain
\[
\mathcal H_{M^\bullet_{11}(k)}(t)=\frac{(1-t^{2k+1})(1+t)}{1-t}=1+2(t+t^2+\cdots+t^{2k})+t^{2k+1}.
\]
This yields $\#M^\bullet_{11}(k)=\mathcal H_{M^\bullet_{11}(k)}(1)=4k+2$.
\end{itemize}

\begin{table}[h]
$$\begin{array}{l|l}
\text{Monoid}&\text{Order}\\
\hline
M_\emptyset&\infty\\
\hline
M_{00}(k,r)&2k+2r-2\\
\hline
M_{01}^\circ(k)&4k+2\\
\hline
M_{01}^\bullet(k)&4k\\
\hline
M_{10}^\circ(k)&4k+2\\
\hline
M_{10}^\bullet(k)&4k\\
\hline
M_{11}^\circ(k)&4k\\
\hline
M_{11}^\bullet(k)&4k+2\\
\end{array}$$
\caption{Order of strict $2$-PIMs}\label{tab-order}
\end{table}

\section{Classification of strict 2-PIMs}
This section is dedicated to the classification of strict 2-PIMs. We will check that two monoids  are defined by two  distinct equations of $\E$ if and only if they are non-isomorphic. 

  First of all, the results of the following lemma are a direct consequence of the enumeration calculations in the previous section. For two monoids $\mathrm{M}$ and $\mathrm{N}$, we write $\mathrm{M}\sim \mathrm{N}$ if they are isomorphic.
\begin{proposition}\label{pr-nonisomorphie-enumeration}
\begin{enumerate}
  \item\label{point1} For every $i\in\{01,10,11\}$, $s\in\{\bullet,\circ\}$  and every couple $k,k'$ of positive integers,  we have
    \[
  \mathrm{M}_i^s(k)\sim
  \mathrm{M}_i^s(k')\Leftrightarrow k=k'.
  \]
  \item\label{point2}  For every $i\in\{01,10,11\}$ and every couple $k,k'$ of positive integers,  we have 
  \[\mathrm{M}_i^\bullet(k)\not\sim
  \mathrm{M}_i^\circ(k').
  \]
  \item\label{point3}  For every $i,j\in\{01,10\}$ and every couple $k,k'$ of positive integers,  we have
  \[ 
  \mathrm{M}_i^\bullet(k)\not\sim
  \mathrm{M}_j^\circ(k').
  \]
  \item\label{point4}  For every $i\in\{01,10\}$, $s\in\{\bullet,\circ\}$  and every couple $k,k'$ of positive integers,  we have
  \[  
  \mathrm{M}_i^s(k)\not\sim
  \mathrm{M}_{11}^s(k').
  \]
  \end{enumerate}
\end{proposition}
\begin{proof}
Item (\ref{point1})  comes from the fact that for the two monoids of $\mathrm{M}_i^s(k)$ and $\mathrm{M}_i^s(k')$ to have the same number of elements, it is necessary and sufficient that $k=k'$. Item (\ref{point2}) arises from the fact that the two monoids $\mathrm{M}_i^\bullet(k)$ and $\mathrm{M}_i^\circ(k')$ can never have the same number of elements. The same argument can be used to prove item (\ref{point3}) and (\ref{point4}).
\end{proof}
In the following, we will show that a strict 2-PIM contains a unique involution.
\begin{lemma}\label{lm-unicite-involution}
Let $\mathrm{M}=\langle \Diamond,\boxempty\mid \{\Diamond^2=\Id, \boxempty^2=\boxempty,\eq\}\rangle$ be a strict 2-PIM. Then $\Diamond$  is the only involution (apart from $\mathrm{Id}$) of $\mathrm{M}$.
\end{lemma}
\begin{proof}
Let us assume there is another involution $\Diamond'$ in $\mathrm{M}$. Necessarily, we have $\Diamond'=\boxempty^{\alpha_1}(\Diamond\boxempty)^p\Diamond^{\alpha_2}$ with $\alpha_1,\alpha_2\in\{0,1\}$, $p\in\mathbb N$.
  As $\Diamond'^2=\Id$, we have
  \[
\Diamond'\Diamond'=  \boxempty^{\alpha_1}(\Diamond\boxempty)^p\Diamond^{\alpha_2}\boxempty^{\alpha_1}(\Diamond\boxempty)^p\Diamond^{\alpha_2}=\left\{
  \begin{array}{ll}
  (\Diamond\boxempty)^{2p}=\Id&\mbox{ if }\alpha_1=\alpha_2=0\\
  \boxempty(\Diamond\boxempty)^{2p}=\Id&\text{ if }\alpha_1=1 \text{ and }\alpha_2=0\\
(\Diamond\boxempty)^{2p-1}\Diamond=\Id&\text{ if }\alpha_1=0\text, \alpha_2=1\text{ and }p>0\\
\Diamond^2=\Id&\mbox{ if }\alpha_1=0\text, \alpha_2=1\text{ and }p=0\\
  (\boxempty\Diamond)^{2p+2}=\Id&\text{ if }\alpha_1=1\text{ and }\alpha_2=1.
  \end{array}
  \right.
  \]
 In each of the above cases, the generated monoid is degenerate and monogenic (see Lemma \ref{LemmaCasDegeneres}). This concludes the proof. 
 \end{proof}
 
   \begin{proposition}\label{pr-M00}
   For every $k,r,k',r'\in\mathbb N\setminus \{0\}$, we have
  $$\mathrm{M}_{00}(k,r)\sim \mathrm{M}_{00}(k',r')\Longleftrightarrow k=k'\wedge r=r'$$
  \end{proposition}
  \begin{proof}
 Let $\phi: \mathrm{M}_{00}(k,r)\rightarrow \mathrm{M}_{00}(k',r')$ be an  isomorphism. We define by $\Diamond_{\phi}$ (resp. $\boxempty_{\phi}$) the image of $\Diamond$ (resp $\boxempty$) by $\phi$.
  From Lemma \ref{lm-unicite-involution} we have $\Diamond_\phi=\Diamond$. From Section \ref{sect-size} we have $\#\mathrm{M}_{00}(k,r)=k+r-2$ and $\#\mathrm{M}_{00}(k',r')=k'+r'-2$. As  $\mathrm{M}_{00}(k,r)\sim\mathrm{M}_{00}(k',r')$, we have $k+r=k'+r'$.
 
For some $\alpha_1,\alpha_2\in\{0,1\}$, $p\in\mathbb N$, we have $\boxempty_{\phi}=\boxempty^{\alpha_1}(\Diamond\boxempty)^p\Diamond^{\alpha_2}$.
Four cases arise:
  \begin{enumerate}
  \item If $\alpha_1=\alpha_2=0$, then we have  
  \[
\phi((\Diamond\boxempty)^{k'})=(\Diamond(\Diamond\boxempty)^p)^{k'}=\boxempty(\Diamond\boxempty)^{(p-1)k'}=\boxempty(\Diamond\boxempty)^{(p-1)(k'+r')}=\phi((\Diamond\boxempty)^{k'+r'})
  \]
  \item If $\alpha_1=1$ and $\alpha_2=0$, then we have
  \[
 \phi((\Diamond\boxempty)^{k'})=(\Diamond\boxempty(\Diamond\boxempty)^p)^{k'}=(\Diamond\boxempty)^{k'(p+1)}=(\Diamond\boxempty)^{(k'+r')(p+1)}=\phi((\Diamond\boxempty)^{k'+r'})
  \]
  \item  If $\alpha_1=0$ and $\alpha_2=1$, then we have
  \[
 \phi((\Diamond\boxempty)^{k'})=(\Diamond(\Diamond\boxempty)^p\Diamond)^{k'}=(\boxempty\Diamond)^{pk'}=
(\boxempty\Diamond)^{p(k'+r')}=\phi((\Diamond\boxempty)^{k'+r'})
  \] 
  \item If $\alpha_1=\alpha_2=1$, then we have
  \[
  \phi((\Diamond\boxempty)^{k'})=((\Diamond\boxempty)^{p+1}\Diamond)^{k'}=(\Diamond\boxempty)^{k'p+1}\Diamond=(\Diamond\boxempty)^{(k'+r')p+1}\Diamond=\phi((\Diamond\boxempty)^{k'+r'})
  \]
  \end{enumerate}
    Hence, for each case, $\mathrm{M}_{00}(k,r)$ satisfies  the equation $(\Diamond\boxempty)^{k'}=(\Diamond\boxempty)^{k'+r'}$. Thus, as $\mathrm{M}_{00}(k,r)$ also satisfies  the equation $\eq_{00}(k,r)$, we have from Lemma \ref{lem:eq1+eq1=eq1} that
  $k\leq k'$ and $r'$ is a  multiple of $r$. As $k+r=k'+r'$ we have $k=k'$ et $r=r'$.

  \end{proof}
 Let $\mathrm{M}$ be a strict $2$-PIM and  $\varphi_{\mathrm M}$ the canonical surjection from  $\{\Diamond,\boxempty\}^*$ to $\mathrm M$. Two words $w$ and $w'$ over  $\{\Diamond,\boxempty\}^*$ are equivalent in ${\mathrm{M}}$ if $\varphi_{\mathrm M}(w)=\varphi_{\mathrm M}(w')$.
  A word $w\in \{\Diamond,\boxempty\}^*$ is  \emph{quasi-reduced} if it does not contain any factor  $\Diamond\Diamond$ nor $\boxempty\boxempty$. If $w$ is a word over  $\{\Diamond,\boxempty\}^*$ then  $[w]$ is the smallest equivalent word to $w$ in $\mathrm M_\emptyset$.
 \begin{lemma}\label{lemmotsequiv}
  Let   $w$ and   $w'$ be two  words over $\{\Diamond,\boxempty\}^*$  and $s\in \{\circ,\bullet\}$.
  \begin{enumerate}
  \item\label{p1}
  If $w$ and $w'$ are equivalent in  $\mathrm{M}_{00}(k,r)$ then
   $|[w]|_{\Diamond}-|[w]|_{\boxempty}\equiv |[w']|_{\Diamond}-|[w']|_{\boxempty} \mod 2.$
   \item\label{p2} If $w$ and $w'$ are equivalent in  $\mathrm{M}_{01}^s(k)$ then
   $[w]=\boxempty u\Leftrightarrow [w']=\boxempty u'$.
  \item\label{p3} If $w$ and $w'$ are equivalent in  $\mathrm{M}_{10}^s(k)$ then
   $[w]= u\boxempty\Leftrightarrow [w']= u'\boxempty$.
\end{enumerate}
  \end{lemma}
\begin{proof}
For item \ref{p1}, simply note that equation $eq_{00}(k,r)$ does not change the parity of $|[w]|_{\Diamond}-|[w]|_{\boxempty}$. As $[w]$ is quasi-reduced,  the other two equations are not involved. For item \ref{p2} (resp. item \ref{p3}), simply note that equation $eq_{01}^s(k)$ (resp. $eq_{10}^s(k)$) does not change the first  (resp. last) letter of words $[w]$.
\end{proof}

   \begin{proposition}
   For every $k,k',r\in\mathbb N\setminus \{0\}$, every $i\in \{01,10,11\}$ and every $s\in \{\circ,\bullet \}$, we have
   $$\mathrm{M}_{00}(k,r)\not\sim \mathrm{M}_i^s(k')$$
   \end{proposition}
\begin{proof}
Since $eq^s_{11}(k)\Rightarrow\eq^s_{10}(k)\wedge\eq^s_{01}(k)$, it is sufficient to look only at the $i\in\{10,01\}$ cases. Let us suppose that there exists an isomorphism  $\phi: \mathrm{M}_{00}(k,r)\rightarrow \mathrm{M}_i^s(k')$. We define by $\Diamond_{\phi}$ (resp. $\boxempty_{\phi}$) the image of $\Diamond$ (resp $\boxempty$) by $\phi$.
  Let us suppose $\boxempty_{\phi}=\boxempty^{\alpha_1}(\Diamond\boxempty)^p\Diamond^{\alpha_2}$. Depending on the values of $\alpha_1,\alpha_2, i, s$ and $k'$, we obtain an additional relation from the
  monoid $\mathrm{M}_i^s$. 
  \[
  \begin{array}{|c|c|c|}
  \hline
  &\mathrm{M}_{01}^\circ(k')&\mathrm{M}_{01}^\bullet(k')\\\hline
  \alpha_1=\alpha_2=0&\boxempty(\Diamond\boxempty)^{k'(p-1)}=\boxempty(\Diamond\boxempty)^{(k'+1)(p-1)}\Diamond&\boxempty(\Diamond\boxempty)^{k'(p-1)}=\boxempty(\Diamond\boxempty)^{k'(p-1)}\Diamond\\\hline
  \alpha_1=1\text{ and }\alpha_2=0&(\Diamond\boxempty)^{k'(p+1)}=(\Diamond\boxempty)^{(k'+1)(p+1)}\Diamond&(\Diamond\boxempty)^{k'(p+1)}=(\Diamond\boxempty)^{k'(p+1))}\Diamond\\\hline
  \alpha_1=0\text{ and }\alpha_2=1&(\boxempty\Diamond)^{k'p}=(\boxempty\Diamond)^{(k'+1)p-1}\boxempty&(\boxempty\Diamond)^{k'p}=(\boxempty\Diamond)^{k'p-1}\boxempty\\\hline
  \alpha_1=\alpha_2=1&(\Diamond\boxempty)^{k'p+1}\Diamond=(\boxempty\Diamond)^{(k'+1)p+1}&(\Diamond\boxempty)^{k'p+1}\Diamond=(\boxempty\Diamond)^{k'p+1}\\\hline
  \end{array}
  \]

  \[
  \begin{array}{|c|c|c|}
  \hline
  &\mathrm{M}_{10}^\circ(k')&\mathrm{M}_{10}^\bullet(k')\\\hline
  \alpha_1=\alpha_2=0&\boxempty(\Diamond\boxempty)^{k'(p-1)}=(\Diamond\boxempty)^{k'(p-1)+p}&\boxempty(\Diamond\boxempty)^{k'(p-1)}=(\Diamond\boxempty)^{(k'-1)(p-1)+p}\\\hline
  \alpha_1=1\text{ and }\alpha_2=0&(\Diamond\boxempty)^{k'(p+1)}=(\Diamond\boxempty)^{(k'p+1)}\Diamond&(\Diamond\boxempty)^{k'(p+1)}=(\Diamond\boxempty)^{(k'-1)(p+1))}\Diamond\\\hline
 \alpha_1=0\text{ and }\alpha_2=1&(\boxempty\Diamond)^{k'p}=(\Diamond\boxempty)^{k'p-1}\Diamond&(\boxempty\Diamond)^{k'p}=(\Diamond\boxempty)^{(k'-1)p-1}\Diamond\\\hline
 \alpha_1=\alpha_2=1&(\Diamond\boxempty)^{k'p+1}\Diamond=(\boxempty\Diamond)^{k'p+1}&(\Diamond\boxempty)^{k'p+1}\Diamond=(\boxempty\Diamond)^{(k'-1)p+1}\\\hline
  \end{array}
  \]
It can be seen that in all cases, the additional relation does not preserve parity between $\Diamond$ and $\boxempty$ and therefore contradicts  Lemma \ref{lemmotsequiv}. Then $\phi$ can not be an isomorphism which proves the result.
    
\end{proof}
  \begin{lemma}\label{lem-notiso-10-01}
      Let $k,k'$ be positive integers and $s,s'\in\{\bullet,\circ\}$. We have the two following assertions:
      \begin{enumerate}
          \item the monoid $\mathrm{M}_{10}^{s'}(k')$ is not homomorphic to any digenic submonoid of $\mathrm{M}_{01}^s(k)$,
         \item the monoid $\mathrm{M}_{01}^{s'}(k')$ is not homomorphic to any digenic submonoid of $\mathrm{M}_{10}^s(k)$.
  \end{enumerate}
  \end{lemma} 
 \begin{proof}
If $\phi$ is a  morphism of monoids, we denote by $X_{\phi}$  the image of an element $X$ by the morphism $\phi$.
  
 \begin{enumerate} 
     \item 
     Suppose there is an morphism $\phi:\mathrm{M}_{10}^{s'}(k')\rightarrow\mathrm{M}_{01}^{s}(k)$.  From Lemma \ref{lm-unicite-involution} we have $\Diamond_\phi=\Diamond$. We have 
  $\boxempty_{\phi}=\boxempty^{\alpha_1}(\Diamond\boxempty)^p\Diamond^{\alpha_2}$. The relation of $\mathrm{M}_{10}^{s'}(k')$ is of the form
  \[
  (\Diamond\boxempty)^{k'}=\boxempty (\Diamond\boxempty)^{k'-\epsilon} (\epsilon\in\{0,1\}).
  \]
  This gives
 \[
  (\Diamond\boxempty^{\alpha_1}(\Diamond\boxempty)^p\Diamond^{\alpha_2})^{k'}=\boxempty^{\alpha_1}(\Diamond\boxempty)^p\Diamond^{\alpha_2}(\Diamond(\boxempty^{\alpha_1}(\Diamond\boxempty)^p\Diamond^{\alpha_2}))^{k'- \epsilon}.
  \]
  If $\alpha_1=1$ then there exists a quasi-reduced word equivalent to the right hand side beginning with $\boxempty$ and there exists a quasi-reduced word equivalent to the left hand side beginning with $\Diamond$. 
  If $\alpha_1=0$ and $p>0$ then we can apply a symmetric argument. In both cases, this contradicts point \ref{p2} of Lemma \ref{lemmotsequiv}. If $\alpha_1=0$ and $p=0$ then the image of the monoid by $\phi$ is not digenic.
  \item Suppose  there is a morphism $\phi:\mathrm{M}_{01}^{s'}(k')\rightarrow\mathrm{M}_{10}^{s}(k)$. From Lemma \ref{lm-unicite-involution} we have $\Diamond_\phi=\Diamond$. Let
  $\boxempty_{\phi}=\boxempty^{\alpha_1}(\Diamond\boxempty)^p\Diamond^{\alpha_2}$. The relation of $\mathrm{M}_{01}^{s'}(k')$ is of the form
  \[
  (\Diamond\boxempty)^{k'}=(\Diamond\boxempty)^{k'+\epsilon}\Diamond (\epsilon\in\{0,1\}).
  \]
  This gives
 \[
  (\Diamond\boxempty^{\alpha_1}(\Diamond\boxempty)^p\Diamond^{\alpha_2})^{k'}=(\Diamond\boxempty^{\alpha_1}(\Diamond\boxempty)^p\Diamond^{\alpha_2})^{k'+\epsilon}\Diamond
  \]
  If $\alpha_2=0$ and $p+\alpha_1>0$ then there exists a quasi-reduced word equivalent to the left hand side that ends in $\boxempty$ and there exists a quasi-reduced word equivalent to the right hand side that ends in $\Diamond$.
  If $\alpha_2=1$ and $p+\alpha_1>0$ then we can apply a symmetric argument. In both cases, this contradicts point \ref{p3} of Lemma \ref{lemmotsequiv}. If $\alpha_1=0$ and $p=0$ then the image of the monoid by $\phi$ is not digenic.

 \end{enumerate}
  \end{proof}
\begin{proposition}\label{MineqMj}
      For every $k,k'\in\mathbb N\setminus \{0\}$, every $i\in \{01,10\}$ and every $s,s'\in \{\circ,\bullet \}$, we have the following assertions:
\begin{enumerate}
    \item\label{it1} 
   $\mathrm{M}_{01}^s(k)\not\sim \mathrm{M}_{10}^{s'}(k')$,
      \item \label{it2}
   $\mathrm{M}_i^s(k)\not\sim \mathrm{M}_{11}^{s'}(k')$,
\end{enumerate}
\end{proposition}
\begin{proof}
    \begin{enumerate}
        \item This is a direct consequence of Lemma \ref{lem-notiso-10-01}.
        \item The case where $s=s'$ is solved in Proposition \ref{pr-nonisomorphie-enumeration}. Consider now $s=\circ$ and $s'=\bullet$. According to Lemma \ref{2et3<=>4} and Lemma \ref{lm-eqi+eq1=eqi}, we have $\eq_{11}^\bullet(k')\Rightarrow \eq_{11}^\circ(k')\Rightarrow \eq_{01}^\circ(k')$ (resp. $\eq_{11}^\bullet(k')\Rightarrow \eq_{10}^\circ(k')$). 
  This means that there exists a morphism $\mathrm{M}_{01}^\circ(k')\rightarrow \mathrm{M}_{11}^\bullet(k')$ (resp. $\mathrm{M}_{10}^\circ(k')\rightarrow \mathrm{M}_{11}^\bullet(k')$).
  We deduce that $\mathrm{M}_{11}^\bullet(k')$ cannot be isomorphic to $\mathrm{M}_{01}^\circ(k)$ (resp.  $\mathrm{M}_{10}^\circ(k)$) because otherwise the monoid $\mathrm{M}_{10}^\circ(k)$ (resp. $\mathrm{M}_{01}^\circ(k)$) would be homomorphic to
  $\mathrm{M}_{01}^\circ(k')$ (resp. $\mathrm{M}_{10}^\circ(k')$) which contradicts the  item \ref{it1} of the proposition. Consider now $s=\bullet$ and $s'=\circ$. According to Lemma \ref{2et3<=>4}, we have $\eq_{11}^\circ(k')\Rightarrow \eq_{10}^\circ(k')$ (resp. $\eq_{11}^\circ(k')\Rightarrow \eq_{01}^\circ(k')$). This means that there exists a morphism $\mathrm{M}_{10}^\circ(k')\rightarrow \mathrm{M}_{11}^\circ(k')$ (resp. $\mathrm{M}_{01}^\circ(k')\rightarrow \mathrm{M}_{11}^\circ(k')$). We deduce that $\mathrm{M}_{11}^\circ(k')$ cannot be isomorphic to $\mathrm{M}_{10}^\bullet(k)$ (resp.  $\mathrm{M}_{01}^\bullet(k)$) because otherwise the monoid $\mathrm{M}_{10}^\bullet(k)$ (resp. $\mathrm{M}_{01}^\bullet(k)$) would be homomorphic to
  $\mathrm{M}_{01}^\circ(k')$ (resp. $\mathrm{M}_{10}^\circ(k')$) which contradicts the  item \ref{it1} of the proposition.
    \end{enumerate}
\end{proof}

 \begin{theorem}\label{th-classification}
For all couple of equations $\eq,\eq'$ of $\E$, we have  $\mon{eq}$ is isomorphic to $\mon{eq'}$  if and only if $eq= eq'$.
\end{theorem}
\begin{proof}
    According to Proposition \ref{pr-nonisomorphie-enumeration}, Proposition \ref{pr-M00} and Proposition \ref{MineqMj}, we can conclude.
\end{proof}
\section{Conclusion}
We describe in Table \ref{tab-resume} the classification of the different strict 2-PIMs   depending on the form of the additional equation.

To complete the 2-PIM classification, we now need to look at the case where both generators are idempotent. Many examples of idempotent operations are closure operations (extensive and increasing idempotent). The notion of closure is defined at the level of sets and seems not to be encapsulated in the formalism we developed above. The reason we're particularly interested in these closure operators is that they are frequently used in language theory, particularly for calculating language orbits.  In forthcoming works, we plan to embed the result in the study of several monoids like the ones that appears in \cite{CDHS11,Das19,Das21,Das22,Kur22}.
\bibliography{../COMMONTOOLS/biblio}

\begin{thebibliography}{10}

\bibitem{CDHS11}
Emilie Charlier, Michael Domaratzki, Tero Harju, and Jeffrey~O. Shallit.
\newblock Finite orbits of language operations.
\newblock In Adrian{-}Horia Dediu, Shunsuke Inenaga, and Carlos
  Mart{\'{\i}}n{-}Vide, editors, {\em Language and Automata Theory and
  Applications - 5th International Conference, {LATA} 2011, Tarragona, Spain,
  May 26-31, 2011. Proceedings}, volume 6638 of {\em Lecture Notes in Computer
  Science}, pages 204--215. Springer, 2011.

\bibitem{Cox34}
Harold Scott~Mac{D}onald Coxeter.
\newblock Discrete groups generated by reflections.
\newblock {\em Annals of Mathematics}, 35(3):588--621, 1934.

\bibitem{Cox35}
Harold Scott~Mac{D}onald Coxeter.
\newblock The complete enumeration of finite groups of the form
  $r_i^2=(r_ir_j)^{k_{ij}}=1$.
\newblock {\em Journal of the London Mathematical Society}, s1-10(1):21--25,
  1935.

\bibitem{Das19}
J{\"{u}}rgen Dassow.
\newblock On the orbit of closure-involution operations - the case of formal
  languages.
\newblock {\em Theor. Comput. Sci.}, 777:192--203, 2019.

\bibitem{Das21}
J{\"{u}}rgen Dassow.
\newblock Some remarks on the orbit of closure-involution operations on
  languages.
\newblock {\em Inf. Comput.}, 281:104811, 2021.

\bibitem{Das22}
J{\"{u}}rgen Dassow.
\newblock The orbit of closure-involution operations: the case of boolean
  functions.
\newblock {\em Beitr Algebra Geom}, 63:321–334, 2022.

\bibitem{GJ08a}
Barry Gardner and Marcel Jackson.
\newblock The kuratowski closure-complement theorem.
\newblock {\em New Zealand Journal of Mathematics}, 38:9--44, 2008.

\bibitem{Ham60}
Preston~Clarence Hammer.
\newblock Kuratowski's closure theorem.
\newblock {\em Nieuw Arch. Wisk.}, 8(3):74--80, 1960.

\bibitem{Kur22}
Kazimierz Kuratowski.
\newblock Sur l'op{é}ration {$\overline{A}$} de l'{A}nalysis situs.
\newblock {\em Fundamenta Mathematicae}, 3:182--199, 1922.

\bibitem{Pel84}
David Peleg.
\newblock A generalized closure and complement phenomenon.
\newblock {\em Discret. Math.}, 50:285--293, 1984.

\bibitem{She10}
David Sherman.
\newblock Variations on kuratowski's 14-set theorem.
\newblock {\em Am. Math. Mon.}, 117(2):113--123, 2010.

\end{thebibliography}
\end{document}